\documentclass[hidelinks,onefignum,onetabnum]{siamart250211}

\usepackage{lipsum}
\usepackage{amsfonts}
\usepackage{graphicx}
\usepackage{epstopdf}
\usepackage{algorithmic}
\ifpdf
  \DeclareGraphicsExtensions{.eps,.pdf,.png,.jpg}
\else
  \DeclareGraphicsExtensions{.eps}
\fi

\newsiamremark{remark}{Remark}
\newsiamremark{hypothesis}{Hypothesis}
\crefname{hypothesis}{Hypothesis}{Hypotheses}
\newsiamthm{claim}{Claim}
\newsiamremark{fact}{Fact}
\crefname{fact}{Fact}{Facts}

\headers{An Example Article}{D. Doe, P. T. Frank, and J. E. Smith}

\title{An Example Article\thanks{Submitted to the editors DATE.
\funding{This work was funded by the Fog Research Institute under contract no.~FRI-454.}}}

\author{Dianne Doe\thanks{Imagination Corp., Chicago, IL 
  (\email{ddoe@imag.com}, \url{http://www.imag.com/\string~ddoe/}).}
\and Paul T. Frank\thanks{Department of Applied Mathematics, Fictional University, Boise, ID 
  (\email{ptfrank@fictional.edu}, \email{jesmith@fictional.edu}).}
\and Jane E. Smith\footnotemark[3]}

\usepackage{amsopn}

\usepackage{mathtools}
\usepackage{comment,color}
\usepackage{enumitem}
\usepackage{placeins}

\DeclareMathOperator*{\argmin}{argmin}

\headers{Finite-Time Stabilization via Optimality Principle}{W. Weng, Y. Chitour and P. Mason}

\ifpdf
\hypersetup{
  pdftitle={On Finite-Time Stabilization of Linear System via Optimality Principle},
  pdfauthor={Y. Chitour, P. Mason and W. Weng}
}
\fi

\title{Finite-Time Stabilization of Linear Systems via Optimal Control}

\author{Weihao Weng\thanks{Laboratoire Signaux et Systèmes (L2S), CNRS, Université Paris-Saclay, France (\email{weihao.weng@centralesupelec.fr}
, \email{yacine.chitour@centralesupelec.fr}, \email{paolo.mason@centralesupelec.fr})}
\and Yacine Chitour\footnotemark[1]
\and Paolo Mason\footnotemark[1]
}

\begin{document}

\maketitle

% REQUIRED
\begin{abstract}
This paper presents an optimal control framework for achieving finite-time stabilization of linear systems. By introducing a suitably constructed integral cost function, we derive a new class of nonlinear controllers that guarantee finite-time stability through the application of the optimality principle. The relationship between the resulting optimal control law and the associated value function is analyzed, leading to the derivation of a Hamilton–Jacobi–Bellman (HJB) equation and the study of its regularity properties. Numerical simulations validate the theoretical findings and illustrate the effectiveness of the proposed method. Furthermore, a discussion on estimating the convergence time is provided.

\end{abstract}

% REQUIRED

% REQUIRED

\section{Introduction}
The problem of designing stabilizing feedback laws for dynamical systems in finite time has received sustained attention due to its theoretical significance and broad range of applications in control engineering, robotics, and aerospace systems. Classical stabilization theory ensures asymptotic convergence of system trajectories toward an equilibrium as time tends to infinity (see, e.g., \cite{khalil1996robust,sontag2013mathematical}). However, in many practical contexts, such as attitude control, guidance systems, and constrained optimization, fast or finite-time convergence is highly desirable, motivating the study of finite-time stability (FTS), wherein trajectories reach equilibrium in a finite duration.

%The seminal works of {\color{red}Y rajouter une ref pour Utkin et une pour Levant} {\color{violet} Utkin \cite{utkin1977variable}, Levant \cite{levant2003higher}}, and
%Bhat and Bernstein \cite{bhat2000finite}.

While the pioneering work %of Utkin 
\cite{utkin1977variable} established finite-time convergence to a manifold via discontinuous control laws, %Levant 
\cite{levant1998robust,levant2003higher} bridged the gap by achieving finite-time stabilization of both sliding variables and their derivatives using continuous control inputs, cf. the concept of higher-order sliding mode. Then homogeneous system theory, as a powerful technical tool, has finally been systematically considered for finite-time stabilization, see for instance \cite{bhat2000finite}
where the definitive Lyapunov-based mathematical conditions for the finite-time stabilizability of continuous autonomous systems are established
and also \cite{hong2002finite} for a classical explicit construction of FTS controllers.
Since then, finite-time and fixed-time control have evolved into a rich research area, encompassing both continuous and discontinuous feedback mechanisms~\cite{cruz2017homogeneous,polyakov2013finite, polyakov2015finite,shtessel2014sliding} %(Hong \cite{hong2002finite}; Polyakov et al. \cite{polyakov2013finite, polyakov2015finite}; Shtessel et al. \cite{shtessel2014sliding}; Cruz-Zavala and Moreno \cite{cruz2017homogeneous}). 
and guaranteeing FTS through the design of feedback laws with prescribed convergence rates~\cite{andrieu2008homogeneous,nakamura2009homogeneous,polyakov2020generalized,rosier1992homogeneous}. %(Rosier \cite{rosier1992homogeneous}; Andrieu et al. \cite{andrieu2008homogeneous}; Nakamura et al. \cite{nakamura2009homogeneous}; Polyakov \cite{polyakov2020generalized}). 
For instance, homogeneous controllers have been effectively applied to the stabilization of perturbed chains of integrators~\cite{doi:10.1137/19M1285937, Harmouche02122017}. %(Harmouche et al. \cite{Harmouche02122017}; Chitour et al. \cite{doi:10.1137/19M1285937}).

Parallel to these developments, optimal control theory provides a systematic framework for deriving feedback laws via the minimization of suitable performance indices under system dynamics constraints. The connection between stabilization and optimal control is made explicit through the Hamilton–Jacobi–Bellman (HJB) equation, whose value function characterizes the minimal cost to reach equilibrium. Finite-time stabilization can thus be viewed as an inverse optimal control problem in which the controller minimizes a cost functional ensuring finite termination (see, e.g., \cite{7155501, NAKAMURA2013552,nakamura2010global}). 
%Nakamura et al. \cite{nakamura2010global} and \cite{7155501}). 
However, the proposed cost functionals often lack flexibility due to their complex structure and absence of intuitive justification linking the minimizing control to the cost itself. %(see, e.g., Nakamura et al. \cite{NAKAMURA2013552, nakamura2010global}). 
Moreover, deriving explicit controllers from the optimality principle remains challenging, even for linear dynamics, particularly beyond the linear–quadratic (LQR) setting.

One of the objectives of the paper consists in designing finite-time stabilizers and feedback laws via optimal control for a linear control system of the form 
$\dot x=Ax+B u$ where $(A,B)$
is a controllable pair. Relying on the corresponding Brunovsky form, it is immediate to see that it is enough to consider the linear single-input control system $\dot x=J_nx+e_n u$ with $x\in\mathbb{R}^n$, $J_n$ is the Jordan block of size $n$ corresponding to the zero eigenvalue, $e_n$ %=(0\cdots 0 \ 1)^T$
is the $n$-th element of the canonical basis of $\mathbb{R}^n$, and $u\in\mathbb{R}^n$. 
This is why we will focus on such a control system for the whole paper. The main idea consists in tailoring appropriate cost functions on an infinite time horizon and establishing a feedback form of the optimal control through the regularity properties of the value function. Finite-time convergence is then obtained by choosing cost functions with homogeneity properties. 

A first class of instantaneous cost functions we consider is of the type $\frac{|u|^q}q+F(x)$ with $q\in (1,\infty)$ and $F$ is convex and positive definite. The optimal control problem consists in minimizing the corresponding integral cost %cost function integrated 
over an infinite time horizon, among all controls $u\in L^q(\mathbb{R}_{\geq 0},\mathbb{R})$. Applying standard approximations of the infinite time horizon problem with finite time horizon problems coupled with ad hoc estimates, we prove existence and uniqueness of the minimizer for the infinite time horizon as well as a characterization of the optimal trajectory via the application of the Pontryagin Maximum Principle (PMP). We then study the regularity properties of the value function $V_\infty$ associated with the infinite-time-horizon problem and prove that it has (essentially) the same regularity as the function $F$. It allows us to first establish that $V_\infty$ satisfies a HJB equation valid over $\mathbb{R}^n$ and then to derive comparison results for super and sub-solutions of the HJB equation, yielding also that $V_\infty$ is the unique solution of that equation among positive definite continuously differentiable functions. As a byproduct, we prove that several feedback laws proposed in the literature (cf.\cite{Harmouche02122017,hong2002finite}), yielding finite-time convergence to the origin, can be associated with the present optimal control framework. 

The feedback laws proposed previously for $q\in (1,\infty)$ turn out to be unbounded over the state space $\mathbb{R}^n$. To get (uniformly) bounded feedback laws, one first relies on the %equation \eqref{eq:u-inf-V-inf} relating 
expression of the optimal control $u_\infty$ in terms of the value function $V_\infty$ for a fixed $q\in (1,\infty)$, namely, 
$u_\infty(x) = - \left\lfloor\frac{\partial V_\infty}{\partial x_n} (x)\right\rceil^{\frac{1}{q-1}}$. One is tempted to let $q$ tend to infinity to get somehow a bounded feedback. This intuition is confirmed by considering the optimal control problem with simply $F$ as instantaneous cost (integrated over an infinite time horizon), to be minimized over all controls $u$ taking values in $[-1,1]$. We can recover most of the results established in the case of $q$ finite, showing in particular that the above optimal control problem is the limit as $q$ tends to infinity of the optimal control problems defined for finite $q$ in the sense that there is convergence of their value functions and optimal controls. 

We close the paper with a section devoted to numerical simulations for both finite and infinite $q$, validating the theoretical findings and showing the flexibility of the proposed approach. Possible extensions to the present work include, for instance, fixed-time stabilization, robustness analysis under model perturbations, and more extensive numerical investigations (in particular in the case of uniformly bounded feedbacks). Another line of research would be to seek solutions to HJB inequalities among predefined classes of positive definite functions (polynomials, piecewise continuous)
in the spirit of Proposition~\ref{cor7}.

The structure of the paper goes as follows. Section~\ref{sec:prob} formulates the problem for finite $q$ and presents preliminary results. Section~\ref{sec:hamiltonian} develops the Hamiltonian approach first for finite-time horizon problems and then for infinite-time horizon problems. Section~\ref{sec:regularity} gathers the regularity results we obtain for the value function $V_\infty$, while in Section~\ref{sec:HJB} we provide the proof that $V_\infty$ satisfies a HJB equation and also derive comparison results with and sub- and super-solutions. Then we provide in Section~\ref{sec:FTS} our findings on finite-time stabilization under homogeneity on the function $F$, and we extend in Section~\ref{sec:bounded} all the previous results with bounded feedback laws. Finally, Section~\ref{sec:simu} provides numerical simulations that illustrate and validate the theoretical results.

\subsection{Notations}
Given an Euclidean space $\mathbb{R}^n$, we use $\|\cdot\|$ to denote the corresponding norm, as well as the induced matrix norm, i.e., $\|M\|=\sup_{\|v\|=1} \|Mv\|$ where $M \in \mathbb{R}^{n\times n}$ and $v\in\mathbb{R}^n$. We use $(e_i)_{1\leq i\leq n}$ to denote the standard canonical basis of $\mathbb{R}^n$.

For $p > 0$ and $I$ an interval of $\mathbb{R}$, $\|\cdot\|_{L^{p}(I)}$ denotes the $L^p$-norm over $L^p$-integrable functions defined on $I$. If $sgn$ denotes the usual set-valued sign function (i.e., $sgn(x)=\frac{x}{|x|}$ for $x\neq 0$ and $sgn(0)=[-1,1]$) and $r$ is a nonnegative real number, the generalized power function $\lfloor\cdot\rceil^r$ is defined as $|\cdot|^rsgn(\cdot)$. For $r=(r_i)_{1\leq i\leq n}$, the previous notation is extended to vector power functions on $\mathbb{R}^n$ with $\lfloor\cdot\rceil^r: x\mapsto (\lfloor x_i\rceil^{r_i})_{1\leq i\leq n}$. %Finally, given $r = (r_i)_{1 \leq i \leq n} \in \mathbb{R}^n$, set $r_- \coloneq \min_{1\leq i\leq n} r_i$ and $r^+ \coloneq \max_{1\leq i \leq n} r_i$. 
A $\mathcal{K}_{\infty}$-function $\alpha:\mathbb{R}_{\geq 0}\to \mathbb{R}_{\geq 0}$ is a continuous strictly increasing  and unbounded function such that $\alpha(0)=0$.

\section{Statement of the problem and preliminary results}
\label{sec:prob}
 Let $n \in \mathbb{N}_{\geq 1}$ and a real number $q\in (1,\infty)$.
Consider the control system
\begin{equation}
\label{n-integrator}
    \dot{x}(t)=J_nx(t)+e_n u(t) 
\end{equation}
evolving for $t\geq 0$ in $\mathbb{R}^n$, where $J_n$ is the Jordan block of size $n$ corresponding to the zero eigenvalue and the input $u$ belongs to $L^{q}(\mathbb{R}_{\geq 0},\mathbb{R})$.
Introduce the instantaneous cost
\begin{equation}\label{eq:inst-cost}
\mathcal{L} : (x,u) \in \mathbb{R}^n\times\mathbb{R} \mapsto \frac{|u|^{q}}{q} +F(x),
\end{equation}
where $F$ is a mapping from $\mathbb{R}^n$ to $\mathbb{R}_{\geq 0}$ satisfying the following standing assumption:
\begin{enumerate}[label= (A\arabic*), leftmargin=1cm]
    \item \label{ass1} $F$ 
     is continuously differentiable,
    convex and positive definite, i.e., $F(0) = 0$ and $F(x)>0$ for every $x\neq 0$.
\end{enumerate}

For $T\in (0,\infty]$, the optimal control problem \((OCP)_T\) consists in minimizing the integral cost 
\begin{align}
\mathcal{J}_T:\mathbb{R}^n\times L^{q}([0,T),\mathbb{R})&\longrightarrow [0,\infty]\nonumber\\
(x_0,u) \longmapsto \int_0^T &\mathcal{L}(x(t),u(t))dt,\label{eq:JT}
\end{align}
for trajectories of~\eqref{n-integrator} starting at any $x_0\in\mathbb{R}^n$ and we set 
\[V_T(x_0) = \operatorname*{inf}_{u \in L^{q}([0,T),\mathbb{R})} \mathcal{J}_T(x_0,u).
\]
The function $V_T$ is the \emph{value function} associated with \((OCP)_T\). It is clear that $\mathcal{J}_{T_1}(x_0,u)\leq \mathcal{J}_{T_2}(x_0,u)$ for every $T_1<T_2$ in $(0,\infty]$ and $u\in L^{q}([0,T_2),\mathbb{R})$, and as a consequence $V_{T_1}\leq V_{T_2}$.

We next establish some basic properties concerning the solutions and the value function associated with $(OCP)_T$, and recall some classical notions of homogeneity and related properties that will be useful in the upcoming sections.
\subsection{Basic properties of $(OCP)_T$}

From the convexity of the instantaneous cost $\mathcal{L}$ with respect to the control variable, it easily follows that the integral cost $\mathcal{J}_T$ is also convex, as stated below.
\begin{lemma}
\label{p:conv}
For $T\in \mathbb{R}_{> 0}$ the function  $\mathcal{J}_T$ is convex with respect to the first variable and strictly convex with respect to the second one.
 If $T = \infty$ the same result holds on the domain of $\mathcal{J}_T$, i.e., the subset of $\mathbb{R}^n\times L^{q}([0,T))$ such that $\mathcal{J}_T$ is bounded.
\end{lemma}
Based on the previous lemma, we get the following theorem.
\begin{theorem}
\label{theorem1}
Let $T\in (0,\infty]$.  
Then, for every $x_0\in\mathbb{R}^n$,  $(OCP)_T$ admits a unique minimizer in $L^{q}([0,T),\mathbb{R})$.  Furthermore, $V_T$ is a positive definite convex function, and it is strictly convex if $T=\infty$. 
\end{theorem}
\begin{proof}
Existence and uniqueness of the minimizer follows from~\cite[page 215 and Theorem 11]{lee1967foundations}. If $x_0\neq 0$, the optimal trajectory starting at $x_0$ (an absolutely continuous function) is nonzero on a nontrivial time interval so that, since $F$ is positive definite, $V_T(x_0)>0$. We deduce that $V_T$ is positive definite. In order to prove the convexity of $V_T$, let $u^{(1)},u^{(2)}\in L^q([0,T))$ be the optimal controls corresponding to two distinct initial conditions $x^{(1)},x^{(2)}\in \mathbb{R}^n$.
Then,  from the convexity of $F$ and the linearity of the dynamics,
\begin{align}
V_T(\alpha x^{(1)}+(1-\alpha) x^{(2)}) & \leq \mathcal{J}_T(\alpha x^{(1)}+(1-\alpha) x^{(2)},\alpha u^{(1)}+(1-\alpha) u^{(2)})\nonumber\\
& \leq \alpha \mathcal{J}_T( x^{(1)}, u^{(1)}) + (1-\alpha) \mathcal{J}_T( x^{(2)},u^{(2)})\nonumber\\
& = \alpha V_T( x^{(1)}) + (1-\alpha) V_T( x^{(2)}),\label{VTconvex}
\end{align}
for every $\alpha\in (0,1)$, that is, $V_T$ is convex. In the case $T=\infty$ one observes that $u^{(1)},u^{(2)}$ are two distinct elements of $L^q$ since otherwise the difference between the optimal trajectories would be given by the diverging function  $t\mapsto e^{J_n t}(x^{(1)}-x^{(2)})$, contradicting the fact that both trajectories must converge to zero due to the boundedness of $V_\infty(x^{(1)}), V_\infty(x^{(2)})$. As $u^{(1)}\neq u^{(2)}$ and by the strict convexity of $\mathcal{J}_\infty$ with respect to $u$ one then deduces that the second inequality in~\eqref{VTconvex} is strict for every $\alpha\in(0,1)$, implying the strict convexity of the map $V_\infty$. 
\end{proof}
Recalling that positive definite convex functions are continuous and radially unbounded, we have the following result (see, e.g.,~\cite[Lemma 4.3]{khalil2002nonlinear}).
\begin{corollary}
\label{lemma-boundkinfty}
Given $T>0$ there exist $\alpha_-,\alpha_+\in \mathcal{K}_\infty$ such that
\[\alpha_-(\|x_0\|)\leq V_T(x_0) \leq \alpha_+(\|x_0\|).\]
Furthermore, the function $\alpha_+$ may be assumed to be independent of $T$, i.e., 
\[V_T(x_0) \leq V_S(x_0) \leq \alpha_+(\|x_0\|),\qquad \forall S\in (T,\infty].\]
\end{corollary}

\subsection{Generalized homogeneity}
We recall below the notion of homogeneity with respect to a family of dilations (see e.g.~\cite{polyakov2020generalized} for an extensive overview). Dilations are defined as follows.
\begin{definition}
    \label{def1}
    Let $m \in \mathbb{N}_{\geq 1},s=(s_1,...,s_m) \in \mathbb{R}_{> 0}^m$,  and $\epsilon > 0$. We define $\delta_\epsilon^s$ as the $s$-weighted dilation from $\mathbb{R}^m$ to itself given by
    \[\delta_\epsilon^s(x_1,x_2,...,x_m) = (\epsilon^{s_1}x_1, \epsilon^{s_2}x_2,...,\epsilon^{s_m}x_m).\]
\end{definition}
Homogeneity for functions and vector fields is then defined as follows. 
\begin{definition}
    \label{def2}
    Let $m \in \mathbb{N}_{\geq 1},s\in \mathbb{R}_{> 0}^m$,  and  $\mu \in \mathbb{R}$. A mapping $f : \mathbb{R}^m \to \mathbb{R}$ is said to be homogeneous of degree $\mu$ with respect to the family of dilations $(\delta^s_\epsilon)_{\epsilon > 0}$ if $f(\delta_\epsilon^s(x))= \epsilon^{\mu} f(x)$ for all $\epsilon > 0$ and $x \in \mathbb{R}^m$. 
    
    A vector field $\mathcal{G}:\mathbb{R}^m \rightarrow \mathbb{R}^m$ is said to be  homogeneous of degree $\mu$ with respect to the family of dilations $(\delta^s_\epsilon)_{\epsilon > 0}$ if $\mathcal{G}(\delta_\epsilon^s(x))= \epsilon^\mu\delta^s_\epsilon(\mathcal{G}(x))$ for all $\epsilon > 0$ and $x \in \mathbb{R}^m$.
\end{definition}

The following classical result (see, e.g., \cite[Theorem 3.2]{hestenes1966calculus} for a more general formulation) expresses the fact that continuous positive-definite functions which are homogeneous with respect to the same family of dilations are commensurable.
\begin{lemma}
    \label{lem1}
     Let $m \in \mathbb{N}_{\geq 1}$ and $W_1, W_2$ be two positive definite and continuous mappings from $\mathbb{R}^m$ to $\mathbb{R}_{\geq 0}$. Suppose 
    there exist $s \in \mathbb{R}_{> 0}^m$ and $d_1,d_2 >0$ such that $W_1$ and $W_2$ are homogeneous of degree $d_1$ and $d_2$, respectively, with respect to the family of dilations $(\delta_{\epsilon}^s)_{\epsilon>0}$.

     Then, there exist $c_1,c_2 > 0$ such that
     \[c_1W_1^\frac{d_2}{d_1}(x) \leq W_2(x) \leq c_2W_1^\frac{d_2}{d_1}(x), \qquad \forall x \in \mathbb{R}^m.\]
     In particular, there exists $c_1,c_2 > 0$ such that, for every $x \in \mathbb{R}^m$, $c_1\sum_{i=1}^m |x_i|^{\frac{d_2}{s_i}}\leq W_2(x) \leq c_2\sum_{i=1}^m |x_i|^{\frac{d_2}{s_i}}$. 
    The inequalities on the right remain true even if the assumption that $W_2$ is positive definite is dropped.
\end{lemma}

By applying the previous lemma and the chain rule, we obtain the following result. 
\begin{corollary}
    \label{cor1}
    Let $W\in C^1(\mathbb{R}^m,\mathbb{R})$ be homogeneous of degree $d$ with respect to the family of dilations $(\delta_{\epsilon}^s)_{\epsilon>0}$, for some $d\in\mathbb{R}$ and $s \in \mathbb{R}_{> 0}^m$.
Then, the function $\partial_i W$ is homogeneous of degree $\bar d-s_i$ with respect to $(\delta_{\epsilon}^s)_{\epsilon>0}$ and, for any $1\leq i \leq m$, there exists a constant $c^{(i)}> 0$ such that 
       $|\partial_i W(x)| \leq  c^{(i)}\sum_{j=1}^m |x_j|^{\frac{\bar d-s_i}{s_j}}$ for all $x \in \mathbb{R}^m$. 
\end{corollary}

\section{Hamiltonian approach}\label{sec:hamiltonian}

According to Theorem~\ref{theorem1}, the problem $(OCP)_T$ with $T\in(0,\infty]$ admits a unique minimizer for every given initial condition. 
In the finite-horizon case (i.e., on $[0, T]$ with $T<\infty$), such minimizers can be studied thanks to the Pontryagin Maximum Principle (PMP), cf.~\cite{pontryagin2018mathematical}.
For that purpose, let us define over $\mathbb{R}^n\times \mathbb{R}\times\mathbb{R}^n\times\mathbb{R}_{\geq 0}$ the Hamiltonian of the problem 
\begin{equation}\label{eq:ham}
H(x,u,p,\lambda) := p^\top J_n x +up_n -\lambda\left(\frac{|u|^{q}}{q}+F(x)\right),
\end{equation}
and consider the corresponding Hamiltonian system
\begin{equation}
\begin{aligned}
    \dot{x}&=\frac{\partial H}{\partial p}(x,u,p,\lambda), \\
    \dot{p}&=-\frac{\partial H}{\partial x}(x,u,p,\lambda). \label{eq-Ham}
\end{aligned}
\end{equation}
In particular, the first equation coincides with the dynamics~\eqref{n-integrator}.
The PMP asserts that, for every initial condition $x_0\in\mathbb{R}^n$, the optimal trajectory $x^{(T)}$ admits an \emph{extremal} (or \emph{optimal}) \emph{lift} $(x^{(T)},p^{(T)})$, i.e., there exists an absolutely continuous function $p^{(T)}$ taking values in $\mathbb{R}^n$ and a constant value $\lambda\in\lbrace 0, 1 \rbrace$  such that $(p^{(T)},\lambda)$ is not identically equal to zero, the pair $(x^{(T)},p^{(T)})$ is a solution of \eqref{eq-Ham} for some control input $u^{(T)}$ satisfying
\begin{equation}
u^{(T)}(t)\in \argmin_{u\in\mathbb{R}} H(x^{(T)}(t),u,p^{(T)}(t)),\qquad \mbox{a.e. }t\in [0,T],\label{eq:min}
\end{equation}
and, moreover,
\begin{equation}
p^{(T)}(T)=0. \label{eq-transversality}
\end{equation}

Every solution $(x,p)$ of~\eqref{eq-Ham} verifying the properties above is called an \emph{extremal pair}.

Since the Hamiltonian is a concave function with respect to the control variable,~\eqref{eq:min} is equivalent to 
\begin{equation*}
 \frac{\partial H}{\partial u}(x^{(T)}(t),u^{(T)}(t),p^{(T)}(t),\lambda)=0 \iff \lambda u^{(T)}(t)=\lfloor  p^{(T)}_n(t)\rceil^\frac{1}{q-1}.
\end{equation*}
The case $\lambda=0$ can then be excluded easily since, in that case,~the dynamics of the variable $p$ reduces to $\dot{p}^{(T)}(t)=-J_n^\top p^{(T)}(t)$
which, together with \eqref{eq-transversality}, implies that $p^{(T)}\equiv 0$, contradicting the nontriviality of $(p^{(T)},\lambda)$. 
Hence, optimal lifts verify the previous conditions with $\lambda=1$, and the corresponding optimal control is given by 
\begin{equation}
    \label{eq:opt}
u^{(T)}=\lfloor  p^{(T)}_n\rceil^\frac{1}{q-1}.\end{equation}
In particular, optimal pairs are solutions of
\begin{equation}
\begin{aligned}\label{eq-Ham2}
   \dot{x} = J_n x+ \lfloor p_n\rceil^\frac{1}{q-1} e_n, \quad 
    \dot{p} =-J_n^\top p+\nabla F(x).
\end{aligned}
\end{equation}

Whenever $q>2$ or $\nabla F$ is not locally Lipschitz, Equations~\eqref{eq-Ham2} do not necessarily satisfy the usual uniqueness assumptions for solutions of ordinary differential equations. 
However, we establish below a uniqueness result for solutions of~\eqref{eq-Ham2} under the additional terminal condition~\eqref{eq-transversality}.

\begin{proposition}
    \label{prop5}
    Let $T > 0$ and $x_0 \in \mathbb{R}^n$. Then, the equations~\eqref{eq-Ham2} together with~\eqref{eq-transversality} admit a unique solution. 
\end{proposition}
\begin{proof}
    Let $T >0$ and $x_0 \in \mathbb{R}^n$. The PMP guarantees the existence of a solution of the equations~\eqref{eq-Ham2} satisfying~\eqref{eq-transversality}. Suppose that $(\bar{x},\bar{p})$ and $(x,p)$ are two solutions. Setting $\Phi(t):=(x(t)-\bar x(t))^\top (p(t)-\bar p(t))$, a direct computation shows that
   \begin{equation*}
 \dot  \Phi =  (\lfloor p_n \rceil^{\frac{1}{q-1}}-\lfloor \bar{p}_n \rceil^{\frac{1}{q-1}})(p_n-\bar{p}_n) + (x-\bar x)^\top (\nabla F(x) - \nabla F(\bar x)).
   \end{equation*} 
    The first term on the right-hand side of the previous equality is nonnegative independently of $p_n,\bar{p}_n$, 
while the second term is always nonnegative as a consequence of the convexity of $F$, so that $\dot \Phi\geq 0$. Moreover $\Phi(0)=\Phi(T)=0$, from which we deduce that $\dot \Phi$ is identically equal to zero on $[0,T]$. We deduce that $p_n=\bar{p}_n$, i.e., the extremal trajectories correspond to the same control input. In turn, this implies that $x=\bar x$ and, from~\eqref{eq-Ham2}, that $\frac{d}{dt}(p-\bar p) = -J_n^\top (p-\bar p)$. By~\eqref{eq-transversality} we deduce that $p=\bar p$, concluding the proof of the proposition.
\end{proof}

We next adapt the PMP conditions to the infinite-time horizon problem. 

\begin{theorem}
\label{theorem2}
Given $x_0\in\mathbb{R}^n$, let $u^{\infty}$ and $x^{\infty}$ be the optimal control and associated trajectory, respectively, for $(OCP)_\infty$ starting from $x_0$. Then 
\begin{enumerate}
\item $x^{\infty}$ admits a unique extremal lift $(x^{\infty},p^{\infty})$, i.e., a unique solution of the Hamiltonian dynamics \eqref{eq-Ham2}, and $u^{\infty} =  \lfloor p^{\infty}_n\rceil^\frac{1}{q-1}$.
\item There exists a $\mathcal{K}_{\infty}$-function $C(\cdot)$ such that $\|(x^{\infty},p^{\infty},u^{\infty})\|_\infty \leq C(\|x_0\|)$ and, moreover, $\lim_{t\to \infty}(x^{\infty}(t),p^{\infty}(t),u^{\infty}(t)) = (0,0,0)$.

\end{enumerate}
\end{theorem}

\begin{proof}
We will obtain $(x^{\infty},p^{\infty})$ as uniform limit of the extremal pairs $(x^{(T)},p^{(T)})$ as $T$ tends to infinity. 

By~Corollary~\ref{lemma-boundkinfty} there exists a
$\mathcal{K}_\infty$-function $\alpha_+$ such that 
\[V_T(x) \leq V_\infty(x) \leq \alpha_+(\|x\|)\]
for every $T\in (0,\infty]$. Using the variation of constants formula, one can express the value $x^{(T)}(t)$ for $t\in [0,1]$ (assuming without loss of generality $T\geq 1$) as 
\[
    x^{(T)}(t)=e^{J_n t}x_0+\int_0^t
e^{J_n(t-\tau)}e_n u^{(T)}(\tau)\, d\tau,
\]
so that, by H\"older's inequality and $\|e^{J_n s}\| \leq e^t$ for $s\in [0,t]$, it holds for $t\in [0,1]$
\begin{align}
\label{bound-xT}
\|x^{(T)}(t)\| & \leq e^t \left(\|x_0\|+t^{\frac{q-1}{q}}\|u^{(T)}\|_{L^q}\right)\nonumber\leq e^t \left(\|x_0\|+t^{\frac{q-1}{q}}(q V_T(x_0))^{\frac{1}{q}} \right)\nonumber\\
& \leq e \left(\|x_0\|+(q \alpha_+(\|x_0\|))^{\frac{1}{q}} \right).
\end{align}

Similarly, we prove below that $p^{(T)}(0)$ is bounded by a 
$\mathcal{K}_\infty$-function of $\|x_0\|$. For this purpose, let us write
\[p^{(T)}_n(t) = \varphi(t)^\top \left(p^{(T)}(0)+\int_0^t e^{J_n^\top s}\nabla F(x^{(T)}(s))ds\right),\quad  \varphi(t) := e^{-J_n^\top t} e_n.\]
From~\eqref{eq:opt} and Corollary~\ref{lemma-boundkinfty}, the map $p^{(T)}_n(\cdot)$ must be  bounded in $L^{\frac{q}{q-1}}([0,1],\mathbb{R})$ by a $\mathcal{K}_\infty$-function of $\|x_0\|$, and similarly for $t\mapsto \Phi_0(t):=\varphi(t)^\top\int_0^t e^{J_n^\top s}\nabla F(x^{(T)}(s)) ds$ as a consequence of \eqref{bound-xT} and the fact that  $\nabla F$ is continuous and satisfies $\nabla F(0)=0$.
Let $\Phi:\mathbb{R}^n\to L^{\frac{q}{q-1}}([0,1],\mathbb{R})$ be the linear and continuous map defined as $\Phi(y) :=  \varphi(\cdot)^\top y$. 
Since $(-J_n^\top,e_n^\top)$ is observable, $\Phi$
is injective and hence admits a linear and continuous inverse map from $\Phi(\mathbb{R}^n)$ to $\mathbb{R}^n$. We get that $p^{(T)}(0) = \Phi^{-1} \left(p^{(T)}_n(\cdot) - \Phi_0(\cdot)\right)$
so that there exists $M>0$ a $\mathcal{K}_\infty$-function $\psi$ for which
\begin{equation}
\label{bound-p0}
\|p^{(T)}(0)\| \leq 
M\left(\|p^{(T)}_n(\cdot)\|_{L^{\frac{q}{q-1}}} + \left\|\Phi_0(\cdot)\right\|_{L^{\frac{q}{q-1}}} \right) \leq \psi(\|x_0\|).
\end{equation}

Uniform boundedness of $(p^{(T)}(0))_{T\geq 1}$ implies the existence of a converging sequence $(p^{(T_k)}(0))_{k\in\mathbb{N}}$, with $\lim_{k\to\infty} T_k = \infty$ so that, by \cite[Theorem 3.2]{hartman2002ordinary}, the sequence $(x^{(T_k)}(\cdot),p^{(T_k)}(\cdot))_{k\in\mathbb{N}}$ converges, uniformly on compact intervals, to a solution $(x^{\infty}(\cdot),p^{\infty}(\cdot))$ of~\eqref{eq-Ham2}. In particular $x^{\infty}$ is a solution of \eqref{n-integrator} starting at $x_0$ associated with a control $u^{\infty}$ satisfying~\eqref{eq:opt}.
Then, for every $T>0$,
\begin{align*}
\mathcal{J}_T(x_0,u^{\infty})& = \lim_{k\to \infty} \mathcal{J}_T(x_0,u^{(T_k)})\leq \lim_{k\to \infty} \mathcal{J}_{T_k}(x_0,u^{(T_k)})= \lim_{k\to \infty}V_{T_k}(x_0) \leq V_\infty (x_0).
\end{align*}
Hence $V_{\infty}(x_0)\leq  \mathcal{J}_\infty(x_0,u^{\infty}) =\lim_{T\to \infty} \mathcal{J}_T(x_0,u^{\infty}) \leq V_\infty (x_0)$,
that is, $(x^{\infty},p^{\infty})$ is an extremal lift associated with the optimal solution $x^{\infty}$ of $(OCP)_\infty$ starting at $x_0$. The uniqueness of the extremal lift follows from the uniqueness of the optimal control $u^{\infty} =  \lfloor p^{\infty}_n\rceil^\frac{1}{q-1}$ and the observability of $(-J_n^\top,e_n^\top)$. Item 1 is proved.

To prove Item 2, first observe that $V_\infty(x^{\infty}(\cdot))$ is non-increasing by optimality of the trajectory $x^{\infty}$, so that, by Corollary~\ref{lemma-boundkinfty},
\[\|x^{\infty}(t)\| \leq \alpha_-^{-1}(V_\infty(x^{\infty}(t))) \leq \alpha_-^{-1}(V_\infty(x_0)) \leq  \alpha_-^{-1}(\alpha_+(\|x_0\|)). \]
Since for any $t\geq 0$, $(x^{\infty}(t +\cdot),p^{\infty}(t +\cdot))$  is the extremal lift associated with the optimal trajectory starting from $x^{\infty}(t)$. Then, by~\eqref{bound-p0},
\[\|p^{\infty}(t)\| \leq \psi(\|x^{\infty}( t)\|) \leq \psi(\alpha_-^{-1}(\alpha_+(\|x_0\|))).\]
By~\eqref{eq:opt} we deduce the existence of a $\mathcal{K}_\infty$-function $C$ such that $\|(x^{\infty},p^{\infty},u^{\infty})\|_\infty \leq C(\|x_0\|)$. To conclude the proof, it is enough to observe that 
\begin{align*}
\lim_{t\to \infty}\|x^{\infty}(t)\| &\leq \lim_{t\to \infty}\alpha_-^{-1}(V_\infty(x^{\infty}(t))) = \alpha_-^{-1}\left( \lim_{t\to \infty} V_\infty(x^{\infty}(t))\right)\\
& =  \alpha_-^{-1}\left( \lim_{t\to \infty} \int_t^\infty \mathcal{L}(x^{\infty}(s),u^{\infty}(s))ds \right)=0
\end{align*}
and that $\lim_{t\to \infty} \|p^{\infty} (t)\|  \leq \lim_{t\to \infty}\psi(\|x^{\infty}( t)\|) =0$.
\end{proof}

\section{Regularity properties of the value function}\label{sec:regularity}
We next study the regularity of the value function associated with 
$(OCP)_\infty$.
\begin{proposition}
    \label{prop14}
    The value function $V_\infty$ is continuously differentiable on $\mathbb{R}^n$ and $\nabla V_\infty(x_0)=-p^{\infty}(0)$, where $(x^{\infty},p^{\infty})$ is the extremal lift associated with the optimal trajectory $x^{\infty}$ starting at $x_0$.
\end{proposition}
\begin{proof}
   Let $K \in \mathbb{R}^{n}$ such that $J_n+e_n K^\top$ is Hurwitz, $x_0, \delta_0 \in \mathbb{R}^n$, and $h \in \mathbb{R}$. 
   
Let $x_K$ be the solution of~\eqref{n-integrator} starting from $x_0+h\delta_0$ associated with $u_K:=u^{\infty}+K^\top(x_K-x^{\infty})$. One has
    \begin{equation}
        V_\infty(x_0+  h\delta_0)  -V_\infty(x_0) \leq \mathcal{J}_\infty(x_0+  h\delta_0,u_K)  -V_\infty(x_0).
\label{eq24-new}
    \end{equation}
    Clearly $x_K(t) = x^{\infty}(t) + h\delta(t)$, where $\delta(t) = e^{(J_n+e_n K^\top)t}\delta_0$ for $t\geq 0$, and, using
    the Lebesgue dominated convergence theorem together with~\eqref{eq:opt} and \eqref{eq-Ham2}, we have 
    \begin{align}
\lim_{h\to 0}\frac1h & \big(\mathcal{J}_\infty(x_0+  h\delta_0,u_K)  -V_\infty(x_0)\big)\nonumber\\
= & \lim_{h\to 0}\int_0^\infty \frac1h \left(\frac{|u_K(t)|^q}{q}-\frac{|u^{\infty}(t)|^q}{q}+F(x_K(t))-F(x^{\infty}(t))\right)dt\nonumber\\
= & \int_0^\infty  \left(\lfloor u^{\infty}(t) \rceil^{q-1}K^\top + \nabla F(x^{\infty}(t))^\top\right)\delta(t)dt\nonumber\\
= & \int_0^\infty  \left(p^{\infty}(t)^\top e_n K^\top + \dot{p}^{\infty}(t)^\top + p^{\infty}(t)^\top J_n\right)\delta(t)dt\nonumber\\
= & \int_0^\infty  \left(p^{\infty}(t)^\top  \dot{\delta}(t) + \dot{p}^{\infty}(t)^\top\delta(t)\right)dt = - p^{\infty}(0)^\top\delta_0,\label{eq25-new}
\end{align}

The subdifferential of the convex function $V_\infty$ at $x\in\mathbb{R}^n$ is defined as the set 
\begin{equation}
\label{eq:subD}
\partial V_\infty(x) = \{v\in\mathbb{R}^n\,:\, v^\top y\leq V_\infty(x+y)-V_\infty(x),\ \forall y\in\mathbb{R}^n\}.
\end{equation}
Then, for every $v\in \partial V_\infty(x_0)$, $\delta_0\in\mathbb{R}^n$, and by using~\eqref{eq24-new} $h v^\top \delta_0 \leq \mathcal{J}_\infty(x_0+  h\delta_0,u_K)  -V_\infty(x_0)$, so that
\begin{align*}
 \lim_{h\to 0^-}\frac{\mathcal{J}_\infty(x_0+  h\delta_0,u_K)  -V_\infty(x_0)}h \leq   v^\top \delta_0 \leq \lim_{h\to 0^+} \frac{\mathcal{J}_\infty(x_0+  h\delta_0,u_K)  -V_\infty(x_0)}h.
\end{align*}       
The limits on both sides exist and are equal to $-p^{\infty}(0)^\top \delta_0$ by~\eqref{eq25-new}. Hence, by arbitrariness of $\delta_0$, it follows that $\partial V_\infty(x_0) = \{-p^{\infty}(0)\}$ implying continuous differentiability of $V_\infty$ according to \cite[Theorem 25.1 and Corollary 25.5.1]{rockafellar1997convex}. As a consequence, $\nabla V_\infty(x_0) = -p^{\infty}(0)$.

\end{proof}
As a consequence of the previous result and of~\eqref{eq:opt}, we have the following result.

\begin{corollary}
\label{cor-feedback}
The optimal control ${u}_\infty$ of $(OCP)_\infty$  is given by
\begin{equation}\label{eq:u-inf-V-inf}
{u}_\infty(x) = - \Big\lfloor\frac{\partial V_\infty}{\partial x_n} (x)\Big\rceil^{\frac{1}{q-1}}.
\end{equation}

\end{corollary}

We next show a H\"older regularity result for the gradient $\nabla V_\infty$ under the assumption that the gradient of $F$ is locally H\"older continuous. Recall that a function $f:\mathbb{R}^n\to\mathbb{R}$ admits a locally $\kappa$-H\"older continuous gradient for some $\kappa \in (0,1)$ if, for every bounded set $U\subset \mathbb{R}^n$, the inequality
    \begin{equation*}
    \|\nabla f(x)-\nabla f(y)\| \leq C\|x-y\|^\kappa, \qquad \forall x,y\in U
    \end{equation*}
    holds for some $C>0$ depending on $U$. 
We need the next two technical lemmas. 
\begin{lemma}
    \label{lem5}
 Let $\kappa \in (0,1)$ and $f \in \mathcal{C}^1(\mathbb{R}^n,\mathbb{R})$. Then, 
 for every bounded set $U\subset \mathbb{R}^n$ there exists $C'>0$ such that 
     \begin{equation}
         \label{eq:holder-2}
       | f(x) - f(y) -  \nabla f(y)^\top (y-x) | \leq C'\|x-y\|^{1+\kappa},\qquad \forall x,y\in U,
    \end{equation}
  if and only if $f$ possesses a locally $\kappa$-H\"older continuous gradient.
\end{lemma}
    
\begin{proof}
    The lemma is proved by straightforward adaptation of the arguments in \cite[Proposition 2.1 and Theorem 4.1]{berger2020quality} .\footnote {The only difference is that in~\cite{berger2020quality}  the values $C, C'$ are actually independent of $U$.}     
\end{proof}

\begin{lemma}
    \label{lem10}
    Let $\mathcal{V}$ be a bounded subset of $\mathbb{R}^2$ and $q>1$. Then there exists $C_{\mathcal{V}}>0$ such that, for every $(v_1,v_2)\in \mathcal{V}$,  the following inequality holds 

    \begin{equation}\label{eq:v1v2}
    \left|\frac{|v_1+v_2|^{q}-|v_1|^{q}}{q}-v_2\lfloor v_1\rceil^{q-1}\right| \leq C_{\mathcal{V}}|v_2|^{\min\{q,2\}}.
     \end{equation}
\end{lemma}
\begin{proof} 
We can clearly assume $v_2\neq 0$.  If $q \geq 2$, we apply the mean value theorem to the $\mathcal{C}^2$ function $t\mapsto \frac{|t|^q}{q}$ and to its derivative to obtain  
\[  \left|\frac{|v_1+v_2|^{q}-|v_1|^{q}}{q}-v_2\lfloor v_1\rceil^{q-1}\right| = \left|v_2(\lfloor v\rceil^{q-1}-\lfloor v_1\rceil^{q-1})\right| = \left|(q-1)v_2(v-v_1) | \tilde{v}|^{q-2}\right| \] 
for some $v,\tilde v$  
satisfying $|\tilde v-v_1|\leq |v-v_1|\leq |v_2|$. Then 
\[  \left|\frac{|v_1+v_2|^{q}-|v_1|^{q}}{q}-v_2\lfloor v_1\rceil^{q-1}\right| \leq (q-1)(|v_1|+|v_2|)^{q-2} |v_2|^2 \] 
which yields \eqref{eq:v1v2} with $C_\mathcal{V} = (q-1) \max_{(v_1,v_2)\in \mathcal{V}}(|v_1|+|v_2|)^{q-2}$. If $q < 2$, we divide both sides of \eqref{eq:v1v2} by $|v_2|^{q}$ and setting $t=v_1/|v_2|$, one is left to prove the inequality
   \[\left|\frac{|t+\mathrm{sign}(v_2)|^{q}-|t|^{q}}{q}-\mathrm{sign}(v_2)\lfloor t\rceil^{q-1}\right| \leq C\]
   for every $t\in \mathbb{R}$, for some $C>0$. By the mean value theorem there exists $t'$ between $t$ and $t+\mathrm{sign}(v_2)$ such that
   \begin{align*}
   \left|\frac{|t+\mathrm{sign}(v_2)|^{q}-|t|^{q}}{q}-\mathrm{sign}(v_2)\lfloor t\rceil^{q-1}\right| &= \left|\lfloor t'\rceil^{q-1}-\lfloor t\rceil^{q-1}\right|\\
    \leq \max&\left\{\lfloor t+1\rceil^{q-1}-\lfloor t\rceil^{q-1},\lfloor t\rceil^{q-1}-\lfloor t-1\rceil^{q-1}\right\}.
   \end{align*}
The right-hand side is bounded by $C = \max_{s\in \mathbb{R}} (\lfloor s+1\rceil^{q-1}-\lfloor s\rceil^{q-1})$, which is well-defined since, for $q\in (1,2)$, $s\mapsto \lfloor s+1\rceil^{q-1}-\lfloor s\rceil^{q-1}$ tends to zero as $s$ goes to $\pm \infty$. This concludes the proof of the lemma.
\end{proof}

\noindent
We finally show the following proposition.
\begin{proposition}
    \label{prop6}
Assume that $\nabla F$ is locally $\kappa$-H\"older continuous.   Then the mapping $\nabla V_\infty$ is locally $\min\lbrace q-1,\kappa\rbrace$-H\"older continuous. 
\end{proposition}
\begin{proof}
 
 Let $x_0,y_0 \in B_R(0)$ be the ball centered at zero of radius $R>0$.  
 Similarly to the proof of Proposition~\ref{prop14}, we take $K\in\mathbb{R}^n$ such that $J_n+e_nK^\top$ is Hurwitz, and let $t\mapsto x_K(t):=x^{\infty}(t)+\delta(t)$ of~\eqref{n-integrator}  associated with the control $u_K:=u^{\infty}+K^\top(x_K-x^{\infty})$, where $\delta(t)= e^{(J_n+e_nK^\top)t}(y_0-x_0)$ for $t\geq 0$. From~\eqref{eq25-new} we have
    \[\int_0^\infty  \left(\lfloor u^{\infty}(t) \rceil^{q-1}K + \nabla F(x^{\infty}(t))\right)^\top\delta(t)\,dt= \nabla V_\infty(x_0)^\top (y_0-x_0),\]
  
   hence it follows from~\eqref{eq24-new} that 
    \begin{align}
     V_\infty(y_0)-V_\infty(x_0)&-\nabla V_\infty(x_0)^\top (y_0-x_0)  \nonumber\\
        &\leq \int_0^\infty
        \Big(\frac{|u_K(t)|^{q}-|u^{\infty}(t)|^{q}}{q} - K^\top \delta(t)\lfloor u^{\infty}(t)\rceil^{q-1})\Big)\, dt
     \nonumber \\
&+\int_0^\infty \Big(F(x_K(t))-F(x^{\infty}(t))-\nabla F(x^{\infty}(t))^\top\delta(t)\Big)\, dt.
     \label{eq:50}
    \end{align}
    
    By Theorem~\ref{theorem2} and the definition of $\delta(\cdot)$ one has
    \[\|(x^{\infty},u^{\infty})\|_\infty \leq C(R), \qquad \|(x_K,u_K)\|_\infty \leq C(R) + 2R (1+\|K\|)\max_{t\geq 0} \|e^{(J_n+e_nK^\top)t}\|.\]
    Using now the local $\kappa$-H\"older continuity of $\nabla F$ and Lemma~\ref{lem10}, one can find a constant $C_1 > 0$ (depending on $R$) such that, for every $t\geq 0$,
    \begin{align}
       |F(x_K(t))-F(x^{\infty}(t))-\nabla F(x^{\infty}(t))^\top\delta(t)| \leq C_1 \|\delta(t)\|^{1+\kappa} \label{eq:51},
    \end{align} 
    \begin{align}\label{eq:52}
 \left|\frac{|u_K(t)|^{q}-|u^{\infty}(t)|^{q}}{q}- K^\top\delta(t)\lfloor u^{\infty}(t)\rceil^{q-1}\right|\leq C_1\|\delta(t)\|^{\min\{2,q\}},
    \end{align}
 
 Using \eqref{eq:50}, \eqref{eq:51} and \eqref{eq:52}, and the exponential convergence of $e^{(J_n+e_nK^\top)t}$ to zero, one deduces   
the existence of $C_2>0$ depending on $R$ such that
\begin{align*}
V_\infty(y_0)&-V_\infty(x_0)-\nabla V_\infty(x_0)^\top (y_0-x_0)\\
&\leq C_1\int_0^{\infty}
\Big(\|\delta(t)\|^
{1+\kappa}+\|\delta(t)\|^{\min\{2,q\}}\Big)\, dt
\leq C_2|y_0-x_0|^{\min\lbrace q,1+\kappa\rbrace}.
\end{align*}

By \eqref{eq:subD}, the left-hand side in the above inequality is nonnegative. We then have that~\eqref{eq:holder-2} holds true for $U=B_R(0)$,  hence for every bounded set $U$. By applying  Lemma~\ref{lem5}, we conclude that $\nabla V_q$ is locally $\min\lbrace q-1,\kappa\rbrace$-H\"older continuous.
\end{proof}

\section{Hamilton-Jacobi-Bellman equation and consequences}
\label{sec:HJB}
In this section, we provide a stationary HJB equation associated with $(OCP)_\infty$, and we prove that this partial differential equation admits a unique solution among positive definite functions $W:\mathbb{R}^n\to\mathbb{R}_{\geq 0}$ of class $\mathcal{C}^1$. 
\begin{theorem}
\label{prop11}
\hspace{-1mm}The value function $V_\infty$ of $(OCP)_\infty$ satisfies the HJB equation
\begin{equation}
\nabla V_\infty(x)^\top J_nx -\frac{q-1}{q}\left|\frac{\partial V_\infty}{\partial x_n}(x)\right|^\frac{q}{q-1} + F(x) = 0,\qquad \forall x \in \mathbb{R}^n. \label{eq:HJB}
\end{equation}
\end{theorem}
\begin{proof}
Let $x^\infty$ be an optimal trajectory for  $(OCP)_\infty$, and $u^\infty$ the corresponding control input. 
Since, for every $t\geq 0$, the curve  $x^{\infty}(t+\cdot)$ corresponds to the solution of $(OCP)_\infty$ starting from $x^{\infty}(t)$, we have 
\begin{equation*}
V_\infty(x^{\infty}(t)) = \int_t^\infty\mathcal{L}(x^\infty(s),u^\infty(s))\,ds
\end{equation*} 
for every $t\geq 0$. Then, the directional derivative of $V_\infty$ along $x^\infty$ satisfies
\[
\dot{V}_{\infty}(x^{\infty}(s)) =  
- \mathcal{L}(x^\infty(s),u^\infty(s)).
\]
By the chain rule and using Corollary~\ref{cor-feedback}, Equation~\eqref{eq:HJB} follows.
\end{proof}
We next derive comparison results regarding sub- and super-solutions of \eqref{eq:HJB}.
\begin{proposition}
    \label{prop13}
    Let $W$ be a positive definite function in $\mathcal{C}^1(\mathbb{R}^n,\mathbb{R}_{\geq 0})$.
    \begin{enumerate}
    \setlength{\itemsep}{5pt}
        \item If $W$ is a sub-solution of \eqref{eq:HJB}, i.e., $W$ verifies
        \[
        \nabla W(x)^\top J_nx -\frac{q-1}{q}\left|\frac{\partial W}{\partial x_n}(x)\right|^\frac{q}{q-1} \leq -F(x), \qquad  \forall x \in \mathbb{R}^n,
        \]
        \noindent
        then $V_\infty(x) \leq W(x)$ for $x \in \mathbb{R}^n$. 
        \item If $W$ is a super-solution of \eqref{eq:HJB}, i.e., $W$ verifies
        \[
        \nabla W(x)^\top J_nx -\frac{q-1}{q}\left|\frac{\partial W}{\partial x_n}(x)\right|^\frac{q}{q-1} \geq -F(x), \qquad \forall x \in \mathbb{R}^n,
        \]
        then $V_\infty(x) \geq W(x)$ for $x \in \mathbb{R}^n$.
    \end{enumerate}
\end{proposition}
\begin{proof}
Let $W$ be as in the statement of the proposition.
 Suppose that $W$ is a sub-solution of \eqref{eq:HJB} and let the feedback control $u_W(x) = -\left\lfloor \frac{\partial W}{\partial x_n}(x) \right\rceil^\frac{1}{q-1}$ for $x\in\mathbb{R}^n$. Since $x\mapsto u_W(x)$ is continuous, \eqref{n-integrator} closed by $u_W$ admits solutions $x_W$ for every initial condition $x_0\in \mathbb{R}^n$ defined on some non trivial time interval of the form $I:=[0,T_*)$ with $0<T_*\leq \infty$. 

 For any $t\in I$, one has 
    \begin{align*}
        W(x_W(t)) &= W(x_0) + \int_0^t  \nabla W(x_W(s))^\top (J_n x_W(s)+e_nu_W(x_W(s))) \, ds  \\
        &= W(x_0) + \int_0^t ( \nabla W(x_W(s))^\top J_n x_W(s) + \frac{\partial W}{\partial x_n}(x_W(s)) u_W(x_W(s)))\, ds \\
        &= W(x_0) + \int_0^t \left( \nabla W(x_W(s))^\top J_nx_W(s) -  \left| \frac{\partial W}{\partial x_n}(x_W(s)) \right|^\frac{q}{q-1} \right)\, ds \\
        &\leq W(x_0) - \int_0^t \left( \frac{|u_W(x_W(s))|^{q}}{q} + F(x_W(s))\right)\, ds.
    \end{align*}
    One gets that for every $t\in I$, 
    \begin{equation}\label{eq:bound W}
    \int_0^t \left(\frac{|u_W(x_W(s))|^{q}}{q} +F(x_W(s))\right)\, ds \leq W(x_0).
    \end{equation}
In particular, $u_W$ belongs to $L^q([0,T_*),\mathbb{R})$.   Now, using the variation of constants formula, we express the value $x_W(t)$ for $t\in [0, T_*)$ as 
\[
    x_W(t)=e^{J_n t}x_0+\int_0^t
e^{J_n(t-\tau)}e_n u_W(x_W(\tau))\, d\tau,
\]
so that, by H\"older's inequality,
$\|x_W(t)\|  \leq e^t \left(\|x_0\|+t^{\frac{q-1}{q}}\|u_W\circ x_W\|_{L^q}\right)$, proving that $T_*=\infty$.
   As $t\to\infty$ in \eqref{eq:bound W} and by definition of $V_\infty$, we get that 
    \begin{align*}
        V_\infty(x_0) \leq \int_0^\infty \left(\frac{|u_W(x_W(s))|^{q}}{q} + F(x_W(s))\right)\, ds \leq W(x_0)
    \end{align*}
    and, by arbitrariness of $x_0$, we conclude that $V_\infty \leq W$.
    
    As for Item $2$, suppose that $W$ is a super-solution of \eqref{eq:HJB}.
    Considering the time derivative of $V_\infty - W$ along the flow of~\eqref{n-integrator} with  $u^{\infty}(x)$ yields 
    \begin{align}
      \dot V_\infty(x)-\dot W(x)
      &= \Big(\nabla V_\infty(x)-\nabla W(x)\Big)^\top \Big(J_n x - e_n\left\lfloor \frac{\partial V_\infty}{\partial x_n}(x) \right\rceil^\frac{1}{q-1} \Big) \nonumber\\
        &=\nabla V_\infty(x)^\top \Big(J_n x-e_n\left\lfloor \frac{\partial V_\infty}{\partial x_n}(x) \right\rceil^\frac{1}{q-1}\Big) \nonumber\\
        &-\nabla W(x)^\top \Big(J_nx-e_n\left\lfloor\frac{\partial W}{\partial x_n}(x)\right\rceil^\frac{1}{q-1} \Big) \nonumber \\
        &+\frac{\partial W}{\partial x_n}(x)
        \Big(\left\lfloor \frac{\partial V_\infty}{\partial x_n}(x)\right\rceil^\frac{1}{q-1}-\left\lfloor \frac{\partial W}{\partial x_n}(x)\right\rceil^\frac{1}{q-1}\Big) \nonumber \\
        &\leq -\frac{1}{q}\left|\frac{\partial V_\infty}{\partial x_n}(x)\right|^\frac{q}{q-1}-\frac{q-1}{q}\left|\frac{\partial W}{\partial x_n}(x)\right|^\frac{q}{q-1} + \frac{\partial W}{\partial x_n}(x)\left\lfloor \frac{\partial V_\infty}{\partial x_n}(x)\right\rceil^\frac{1}{q-1}. \label{eq35}
    \end{align}
    Setting $a:=\left\lfloor \frac{\partial V_\infty}{\partial x_n}(x)\right\rceil^\frac{1}{q-1}$ and 
    $b:=\frac{\partial W}{\partial x_n}(x)$, the right-hand side of the last inequality in \eqref{eq35} simply reads
    $-\frac{1}{q}|a|^{q} -\frac{q-1}{q}|b|^{\frac{q}{q-1}}+ab$, which is nonpositive by  
    Young's inequality. 
    One deduces that, along the optimal trajectory $x^{\infty}$ starting at $x_0\in\mathbb{R}^n$,
    \[
    V_\infty(x^{\infty}(t))-W(x^{\infty}(t))\leq  V_\infty(x_0)-W(x_0),\qquad \forall t\geq 0.
    \]
    By Item~2 of Theorem~\ref{theorem2} and passing to the limit as $t$ goes to infinity, 
    one deduces that $V_\infty(x_0) \geq W(x_0)$ so that, by arbitrariness of $x_0$, the proposition is proved.
\end{proof}
\noindent
From Proposition~\ref{prop13}, we get the following corollary.
\begin{corollary}
    \label{cor6}
    The value function $V_\infty$ is the unique solution of \eqref{eq:HJB} among positive definite functions in $\mathcal{C}^1(\mathbb{R}^n,\mathbb{R})$. Conversely, assume that $W\in \mathcal{C}^1(\mathbb{R}^n,\mathbb{R})$ is a positive definite mapping
    such that \[F_W(x) := \frac{q-1}{q}\left|\frac{\partial W}{\partial x_n}(x)\right|^\frac{q}{q-1}- \nabla W(x)^\top J_nx \]
    satisfies $(A1)$. Then  $W$ is the value function associated with $(OCP)_{\infty}$ with $F=F_W$, and the corresponding optimal $u_W$ is given by $u_W(x)=-\left\lfloor\frac{\partial W}{\partial x_n}(x)\right\rceil^\frac{1}{q-1}$.
    
\end{corollary}
\begin{proof}
        Let $W \in \mathcal{C}^1(\mathbb{R}^n,\mathbb{R})$ be a positive definite function solution of \eqref{eq:HJB}, then $W$ is a super-solution of \eqref{eq:HJB} and also a sub-solution of \eqref{eq:HJB}. Then, by using Proposition~\ref{prop13}, one immediately deduces that $V_\infty(x) = W(x)$ for every $x \in \mathbb{R}^n$. The second part of the corollary is an immediate consequence of the first one.
\end{proof}

\section{Finite-time stabilization}
\label{sec:FTS}

We will prove below that, under the homogeneity assumption~\ref{ass2} on the function $F$ and in addition to the hypothesis~\ref{ass1} introduced above, all solutions of $(OCP)_\infty$ reach the origin in finite time. 

\begin{enumerate}[label= (A\arabic*), leftmargin=1cm]
\setcounter{enumi}{1}
    \item \label{ass2}  Let $d,\mu \in \mathbb{R}_{>0}$ such that $\mu < \frac{d(q-1)}{qn}$ and define $r\in \mathbb{R}_{>0}^n$ by
$r_i =\frac{d}{q}+(n-i+1)\mu$ for $i=1,\dots,n$. 
      Then $F$ is   homogeneous of degree $d$ with respect to the family $(\delta^r_\epsilon)_{\epsilon>0}$, i.e., 
    $F(\delta_\epsilon^r(x)) = \epsilon^dF(x)$ for
    every $\epsilon > 0$ and $x \in \mathbb{R}^n$.
\end{enumerate}
We first have that 
applying the dilation $\delta_\epsilon^r$ to a solution of $(OCP)_\infty^{lim}$ preserves its 
optimality, up to a time rescaling.
\begin{lemma}
\label{lem-hom-opt-0}
Assume that $F$ satisfies \ref{ass2}, in addition to \ref{ass1}. Let $ x^\infty(\cdot)$ be an optimal trajectory for $(OCP)_\infty$. Then, for every $\epsilon>0$, the map $x^\infty_\epsilon(t) := \delta^r_\epsilon (x^\infty(\epsilon^{-\mu}t))$ also corresponds to  an optimal trajectory for $(OCP)_\infty$.  Moreover $V_\infty$ is homogeneous of degree $d +\mu$ with respect to  $(\delta^r_\epsilon)_{\epsilon>0}$. 
\end{lemma}
\begin{proof}
By a straightforward computation, given $\epsilon>0$ and a trajectory $x$ of~\eqref{n-integrator} with associated control $u$, then $x_\epsilon(t):= \delta^r_\epsilon (x(\epsilon^{-\mu}t))$ is also a trajectory of the system associated with the control $u_\epsilon(t):=\epsilon^{d/q} u(\epsilon^{-\mu}t)$. In particular, the map $u\mapsto u_\epsilon$ is a bijection from $L^q(\mathbb{R}_{\geq 0},\mathbb{R})$ to itself and, using the homogeneity assumption on $F$ with a time rescaling, one  gets $\mathcal{J}_\infty(\delta^r_\epsilon (x(0)),u_\epsilon) = \epsilon^{d+\mu} \mathcal{J}_\infty(x(0),u)$. It follows that $x$ is the optimal trajectory with initial condition $x(0)$ if and only if $x_\epsilon$ is the optimal trajectory with initial condition $\delta_\epsilon^r(x(0))$, and $V_\infty(\delta_\epsilon^r(x(0))) = \epsilon^{d +\mu}V_\infty(x(0))$ for every $\epsilon>0$. This concludes the proof of the lemma.
\end{proof}

\begin{remark}
\label{remark-control}
The optimal control for $(OCP)_\infty$ which, according to Corollary~\ref{cor-feedback}, can be expressed in feedback form as ${u}^\infty(x) = - \left\lfloor\frac{\partial V_\infty}{\partial x_n} (x)\right\rceil^{\frac{1}{q-1}}$ 
is homogeneous of degree $\frac{d+\mu-r_n}{q-1} = \frac{d}{q}$ with respect to $(\delta^{(r,s)}_\epsilon)_{\epsilon>0}$ by Lemma~\ref{lem-hom-opt-0} and Corollary~\ref{cor1}.  
\end{remark}
\begin{remark}\label{rem:HF-homogeneity}
It is noteworthy that the Hamiltonian vector field 
\[\mathcal{F}: (x,y) \in \mathbb{R}^n\times\mathbb{R}^n \mapsto \begin{pmatrix}
    J_n x+e_n\lfloor y_n\rceil^\frac{1}{q-1} \\
    -J_n^\top y +\nabla F(x)
\end{pmatrix},\]
corresponding to the dynamics~\eqref{eq-Ham2}, is homogeneous with respect to a family of dilations defined 
as follows. Let $s=(s_1,\cdots,s_n)$ with 
$s_i=  \frac{d(q-1)}{q}-(n-i)\mu$ for $1\leq i\leq n$. 
Then $\mathcal{F}$ is homogeneous of degree $-\mu$ with respect to the family of dilations  
    $(\delta^{(r,s)}_\epsilon)_{\epsilon > 0}$, where $\delta^{(r,s)}_\epsilon(x,y):= (\delta^r_\epsilon(x), \delta^s_\epsilon(y))$ is defined on $\mathbb{R}^n\times\mathbb{R}^n$.
\end{remark}

We consider below the {\it settling time map $x_0\mapsto T_{x_0}$},
where $T_{x_0}$ is the first time so that the optimal trajectory starting from $x_0$ vanishes and stays at the origin for every time $t\geq T_{x_0}$. We have the following finite-time stabilization result. 
\begin{theorem}
    \label{theorem-finitetime}
    Let  $x_0\in \mathbb{R}^n$ and consider the solution $x^{\infty}$ of $(OCP)_\infty$ starting at $x_0$. Then $x^{\infty}$ reaches the origin in finite time. Moreover, the settling time map from $\mathbb{R}^n$ to $\mathbb{R}_{\geq 0}$  is well-defined, 
    homogeneous of degree $\mu$ with respect to $(\delta_\epsilon^r)_{\epsilon>0}$, and 
    $T_{x_0}\leq \frac{d+\mu}{c\mu}\,V_{\infty}^\frac{\mu}{d+\mu}(x_0)$ for every $x_0\in\mathbb{R}^n$, where $c = \max_{x \in V_\infty^{-1}(1)} F(x)$.
\end{theorem}
\begin{proof}
Let $u^\infty$ be the optimal control associated with $x^\infty$.

As in the proof of Theorem~\ref{prop11}, the directional derivative of $V_\infty$ along the optimal trajectory satisfies
\begin{align*}
\dot{V}_{\infty}(x^{\infty}(s))& = 
- \mathcal{L}(x^\infty(s),u^\infty(s))\leq -F(x^{\infty}(s)).
\end{align*}
From the homogeneity of $V_\infty$ given by Lemma~\ref{lem-hom-opt-0}, as well as the positive-definiteness and homogeneity of $F$, we obtain from Lemma~\ref{lem1} that
\begin{equation}\label{eq:diff-ineq0}
\dot{V}_{\infty}(x^{\infty}(s)) \leq -c\, V_\infty^{\frac{d}{d+\mu}}(x^{\infty}(s)),\ \ s\geq 0,\end{equation}
where $c = \max_{x \in V_\infty^{-1}(1)} F(x)$. It implies in particular that $s\mapsto {V}_{\infty}(x^{\infty}(s))$ is non increasing. It follows that, if $t_{x_0}$ is the first time such that $x^{\infty}$ reaches the origin, then ${V}_{\infty}(x^{\infty}(s))=0$ for $s\geq t_{x_0}$. Hence $t_{x_0}=T_{x_0}$. If $T_{x_0}>0$, then $V_{\infty}(x^{\infty}(s))>0$ for $s\in [0,T_{x_0})$ and we can multiply both sides of \eqref{eq:diff-ineq0} by $\frac{\mu}{d+\mu}V_{\infty}^{-\frac{d}{d+\mu}}(x^{\infty}(s))$  
to obtain $\frac{d}{ds}\left(V_{\infty}^\frac{\mu}{d+\mu}(x^{\infty}(s))\right) \leq -c\frac{\mu}{d+\mu}$.
Integrating  both sides over
$[0,T_{x_0}]$  
yields
\[0 = V_{\infty}^\frac{\mu}{d+\mu}(x^{\infty}(T_{x_0}))\leq V_{\infty}^\frac{\mu}{d+\mu}(x_0)-c\frac{\mu}{d+\mu}T_{x_0},\]
so that $T_{x_0}\leq \frac{d+\mu}{c\mu}\,V_{\infty}^\frac{\mu}{d+\mu}(x_0)$.

Finally notice that the homogeneity of $x_0\mapsto T_{x_0}$
follows from the homogeneity property of optimal trajectories of $(OCP)_\infty$, as described in Lemma~\ref{lem-hom-opt-0}.
\end{proof}

\begin{remark}
Theorem~\ref{theorem-finitetime} is reminiscent of ~\cite[Theorem 8.6, page 240]{polyakov2020generalized}, where finite-time stability is shown for vector fields that are homogeneous of negative degree with respect to a family of dilations and under a uniform asymptotic stability assumption.
Here, optimal trajectories are exactly the integral curves of the vector field $F_{opt}$ defined on $\mathbb{R}^n$ by
\[F_{opt}(x)=J_nx- \left\lfloor\frac{\partial V_\infty}{\partial x_n} (x)\right\rceil^{\frac{1}{q-1}}e_n,
\]
which is homogeneous of degree $-\mu$ with respect to $(\delta_\epsilon^r)_{\epsilon>0}$. Thanks to Theorem~\ref{theorem2}, one gets that the origin is an asymptotically stable equilibrium for $\dot x=F_{opt}(x)$.  
To apply \cite[Theorem 8.6, page 240]{polyakov2020generalized}, one would need to prove an additional uniformity property on bounded time intervals, which is not guaranteed by Theorem~\ref{theorem2}.

\end{remark}

The next result asserts that solutions of $(OCP)_\infty$ are also solutions of $(OCP)_T$ for $T$ large enough.

\begin{proposition}
\label{cor3}
Let $x^{\infty}$ and $T_{x_0}$ as in Theorem~\ref{theorem-finitetime} and denote as $(x^{\infty},p^{\infty})$ the associated extremal lift.
Then, $(x^{\infty}(t),p^{\infty}(t))=(0,0)$ for every $t\geq T_{x_0}$. Moreover, for every $T\geq T_{x_0}$ the restriction of $(x^{\infty},p^{\infty})$ to $[0, T]$ corresponds to the unique extremal lift associated with the solution of $(OCP)_{T}$ starting at $x_0$.

\end{proposition}
\begin{proof}
By Theorem~\ref{theorem-finitetime} and Remark~\ref{remark-control}, it follows that the optimal control $u^\infty$ is equal to zero for $t\geq T_{x_0}$.
 
Then $\dot{p}^{\infty}(t) = -J_n^\top p^{\infty}(t)$ and, by~\eqref{eq:opt}, $e_n^\top p^{\infty}(t) = 0$ for every $t\geq T_{x_0}$. Since the pair $(-J_n^\top,e_n^\top)$ is observable, we deduce that $p^{\infty}(t) = 0$ for every $t\geq T_{x_0}$. Then, the last part of the corollary follows from  Proposition~\ref{prop5}.
\end{proof}

The following result is an application of Proposition~\ref{prop13} and may be interpreted as a generalization of Theorem~\ref{theorem-finitetime} and Corollary~\ref{cor-feedback} in the case in which the value function $V_\infty$ of $(OCP)_\infty$ is replaced by a function satisfying a suitable condition generalizing the HJB equation~\eqref{eq:HJB}.
\begin{proposition}
\label{cor7}
Let $W$ be a positive definite mapping in $\mathcal{C}^1(\mathbb{R}^n,\mathbb{R})$ and assume that 
\begin{enumerate}
    \item there exists $C_0>0$ and $m\in (1,\infty)$ such that, the function $F_W$ defined by $F_W(x):=C_0\left|\frac{\partial W }{\partial x_n}(x)\right|^{m}-\nabla W(x)^\top J_nx$
is positive definite; 
\item there exist $\alpha_1$, $\alpha_2$,  
$d>0$ and $\mu\in (0, \frac{d}{mn})$ such that 
        \begin{equation}\label{eq:conv11}
            \alpha_1\sum_{i=1}^n |x_i|^\frac{d}{r_i} \le F_W(x) \le \alpha_2\sum_{i=1}^n |x_i|^\frac{d}{r_i},\qquad\forall x \in \mathbb{R}^n,
        \end{equation}
        where  
        $r_i:= \frac{d(m-1)}{m}+(n-i+1)\mu$ for $1\leq i\leq n$.
\end{enumerate}
Then, the feedback control $u_W(x) = -C_0\left\lfloor\frac{\partial W}{\partial x_n}(x)\right\rceil^{m-1}$ induces finite-time convergence for \eqref{n-integrator}. 
 If, moreover,  $F_W$ is continuously differentiable, convex and homogeneous of degree $d$ with respect to  $(\delta^r_\epsilon)_{\epsilon>0}$, then the feedback control $\tilde{u}_W(x) = -mC_0\left\lfloor\frac{\partial W}{\partial x_n}(x)\right\rceil^{m-1}$ solves $(OCP)_\infty$ with $F = (mC_0)^{\frac{1}{m-1}} F_W$ and $q = \frac{m}{m-1}$ and induces finite-time convergence.
\end{proposition}
\begin{proof}
     Let $W$, $C_0$, and $F_W$ be as in the statement. Let $\alpha:= (mC_0)^{\frac{1}{m-1}}$ and set $q:=\frac{m}{m-1}\in (1,\infty)$. A direct computation shows that 
     $\alpha F_W(x) = \frac{q-1}{q}\left|\frac{\partial (\alpha W)}{\partial x_n}(x)\right|^{\frac{q}{q-1}}-\nabla (\alpha W)(x)^\top J_nx$. 
 
For $i \in \{1,2\}$ set $F_i:= \alpha_i\sum_{j=1}^n |x_j|^\frac{d}{r_j}$, with $\alpha_1,\alpha_2$ as in~\eqref{eq:conv11}. Let $V_{F_1}$ and $V_{F_2}$ be the value functions of the infinite-horizon optimal control problems $(OCP)_\infty$ associated with $F_1$ and $F_2$, respectively.  As $F_W$ satisfies~\eqref{eq:conv11}, applying Proposition~\ref{prop13} one obtains 
$V_{F_1}(x) \le W(x) \le V_{F_2}(x)$ for $x \in \mathbb{R}^n$. Using Lemma~\ref{lem-hom-opt-0} and Lemma~\ref{lem1}, one can find $C>0$ such that 
 \begin{equation}
    F_1(x) \geq C V_{F_2}^\frac{d}{d+\mu}(x), 
    \qquad \forall x \in \mathbb{R}^n.\label{eq:estimate2} 
 \end{equation}
 \noindent
Let $x_W$ be the solution of~\eqref{n-integrator} associated with the continuous feedback law $u_W(x) = -C_0\left\lfloor\frac{\partial W}{\partial x_n}(x)\right\rceil^{m-1}$ starting from any arbitrary $x_0 \in \mathbb{R}^n$. Similarly to the proof of Theorem~\ref{theorem-finitetime}, on the non trivial time interval $I$ where $x_W$ is defined, the time derivative of $W$ along $x_W$ verifies for $t\geq 0$
    \begin{align*}
        \frac{d}{dt}W(x_W(t))=\hspace{-1mm}-F_W(x_W(t)) \le -F_1(x_W(t)) 
        \le - CV_{F_2}^\frac{d}{d+\mu}(x_W(t)) \le -CW^\frac{d}{d+\mu}(x_W(t)),
    \end{align*}

where the successive upper-bounds are obtained using~\eqref{eq:conv11} as well as~\eqref{eq:estimate2}. One gets that $I=\mathbb{R}_{\geq 0}$ and finite-time convergence follows for \eqref{n-integrator}. In particular, as in the proof of Theorem~\ref{theorem-finitetime}, we get that the trajectory starting at $x_0$ reaches the origin in a finite time $T_{x_0}$ satisfying
$T_{x_0}\leq \frac{d+\mu}{c C\mu}\,W^\frac{\mu}{d+\mu}(x_0)$.
The last part of the proposition follows from Corollary~\ref{cor6} and Theorem~\ref{theorem-finitetime}.
\end{proof}

As an application, one can establish a link between the results of  \cite{Harmouche02122017,hong2002finite}  and those of the previous sections. In the sequel, we will only focus on the connection with \cite{hong2002finite} since the one with  \cite{Harmouche02122017} is derived similarly.
First, we recall the expression of the controllers in \cite{hong2002finite}. Assume $n\geq 2$. Let $r_1,k > 0$ be such that $r_1 > (2n-1)k$, and set
$r_i =r_1-(i-1)k$ for $i=2,\dots,n$. Let $ \beta_0=r_2+1,\ \beta_{n-1}=2$, and
$\beta_j=\beta_0\frac{r_1}{r_{j+1}}$ for $j=1,\dots,n-2$.

Let $(\ell_j)_{1 \leq j \leq n}$ be a family of positive constants and set $u_0=0$. For $1 \leq j \leq n$ we set $\bar{x}_j=(x_1,\dots,x_j) \in \mathbb{R}^j$ and we define recursively the following functions
\begin{equation}\label{eq39}
    u_j(\bar{x}_j):= -\ell_j\left\lfloor  \lfloor x_j \rceil^{\beta_{j-1}-1} -\lfloor u_{j-1}(\bar{x}_{j-1})\rceil^{\beta_{j-1}-1} \right\rceil^\frac{r_j-k}{r_j(\beta_{j-1}-1)}, \end{equation}
    \begin{equation}\label{eq40}
    W_j(\bar{x}_j):=\frac{\left(|x_j|^{\beta_{j-1}}+(\beta_{j-1}-1)|u_{j-1}(\bar{x}_{j-1})|^{\beta_{j-1}}\right)}{\beta_{j-1}} 
    -x_j\lfloor u_{j-1}(\bar{x}_{j-1}) \rceil^{\beta_{j-1}-1}. 
\end{equation}
Furthermore, let $V_1(y) :=  |y|^{\beta_0}$ for $y\in\mathbb{R}$ and
\begin{equation}
     V_k(\bar{x}_k) := W_k(\bar{x}_k) + V_{k-1}(\bar{x}_{k-1})^{\frac{\beta_{k-1}r_k}{\beta_{k-2}r_{k-1}}},\qquad k=2,\dots,n. \label{eq41}
\end{equation}
It was shown in \cite{hong2002finite} that the feedback controller $u_n$ stabilizes the integrator chain in finite time provided that the constants $\ell_j$ are large enough. Furthermore, from (\ref{eq40}) and (\ref{eq41}), 
$u_n(x) = -\ell_n\left\lfloor  \frac{\partial V_n}{\partial x_n}(x)\right\rceil^{m-1}$
with $m = \frac{2r_n-k}{r_n}$ and the function
$F(x):=  \ell_n\left|\frac{\partial V_n(x)}{\partial x_n }\right|^{m} -  \nabla V_n(x)^\top J_n x$ 
is continuous, positive definite and homogeneous of degree $d = 2r_n-k$ with respect to the family of dilations $(\delta^r_\epsilon)_{\epsilon >0}$. By Lemma~\ref{lem1} the assumption~\eqref{eq:conv11} is satisfied.
The fact that the feedback $u_n$ stabilizes the system in finite time can then be deduced from Proposition~\ref{cor7}. 

\section{Stabilization in finite time with bounded controls}\label{sec:bounded}
In the previous section we have shown that, under the assumptions~\ref{ass1} and \ref{ass2}, the solutions of $(OCP)_\infty$ reach the origin in finite time and that the corresponding optimal controls can be expressed in feedback form as $u =u^\infty(x)$, where $u^\infty$ is homogeneous of degree $d/q$ with respect to a suitable family of dilations.
 
In particular, these feedback laws are not uniformly bounded over $\mathbb{R}^n$. We are tempted to let the parameter $q$ tend to $\infty$, in the hope of obtaining 0-homogeneous (hence globally bounded, if locally bounded) finite-time stabilizing feedback laws.

\subsection{Definition of the limit optimal control problem}
In order to deal with the case $q=\infty$, we consider for $T\in (0,\infty]$ and $u \in L^{\infty}([0,T),[-1,1])$ the cost function $\mathcal{K}_T(x_0,u) \;=\; \int_0^T F(x(t))\, dt$ 
where $F$ satisfies Assumption~\ref{ass1} and $x(\cdot)$ is the trajectory of~\eqref{n-integrator} starting at $x_0\in\mathbb{R}^n$. The intuition behind this choice comes from the fact that the limit of the instantaneous cost $\mathcal{L}$ of $(OCP)_T$ as $q$ tends to infinity is given by the function
 which associates with every $(x,u)\in \mathbb{R}^n\times\mathbb{R}$, the value $F(x)$ if $|u|\leq 1$ and $\infty$ if $|u|> 1$.
We then define the optimal control problem \((OCP)^{lim}_T\) as
\begin{equation}\label{eq:OCPinf}
(OCP)^{lim}_T:\qquad \min \{\mathcal{K}_T(x_0,u)
\mid u \in L^{\infty}([0,T),[-1,1])\},
\end{equation}
with value function
$V^{lim}_T(x_0) = \operatorname*{inf}_{u \in L^{\infty}([0,T),[-1,1])} \mathcal{K}_T(x_0,u)$.

In the following, we will occasionally use the following assumption, which is slightly stronger than~\ref{ass1}.

\medskip

\begin{enumerate}[label= (A\arabic*)$'$, leftmargin=1cm]
    \item \label{ass1bis} $F$ 
     is continuously differentiable, strictly convex, and positive definite. 
\end{enumerate}

\medskip

We have the following theorem. 
\begin{theorem}
    \label{theorem1-inf}
    Let $T\in (0,\infty]$.   
Then, for every $x_0\in\mathbb{R}^n$,  $(OCP)_T^{lim}$ admits a minimizer in $L^{\infty}([0,T),[-1,1])$ and $V_T^{lim}$ is a positive definite convex function. If $F$  satisfies~\ref{ass1bis}, then $(OCP)_T^{lim}$ has  unique minimizer  and $V_T^{lim}$ is strictly convex.

\end{theorem}
\begin{proof}
We assume $T=\infty$, the case $T<\infty$ being analogous. By weak-$*$ compactness of $L^{\infty}(\mathbb{R}_{\geq 0},[-1,1])$, any minimizing sequence $(u_k)_{k\geq 1}$ for $(OCP)_\infty^{lim}$ 
weak-$^*$ converges, up to a subsequence, to a control $u^*$ in $L^{\infty}(\mathbb{R}_{\geq 0},[-1,1])$. and the corresponding trajectories $(x^{(k)})_{k\geq 1}$ from $x_0$ converge uniformly on compact intervals to the trajectory $x^*$ associated with $u^*$. For every $\varepsilon>0$, $\tau>0$, and $k$ large enough one has 
\[\int_0^\tau F(x^{(k)}(t))dt \leq \int_0^\infty F(x^{(k)}(t))dt\leq V^{lim}_\infty(x_0) +\varepsilon. \]
Passing to the limit as $k$ goes to infinity on the left-hand side and applying the monotone convergence theorem, we obtain 
\[ \int_0^\infty F(x^*(t))dt = \lim_{\tau\to \infty} \int_0^\tau F(x^*(t))dt \leq V^{lim}_\infty(x_0) +\varepsilon.\]
As $\varepsilon>0$ is arbitrary, we conclude that $u^*$ minimizes $(OCP)_\infty^{lim}$.
Positive definiteness and convexity of $V_T^{lim}$ can be shown exactly as in the proof of Theorem~\ref{theorem1}.
The additional properties claimed under the assumption~\ref{ass1bis} can be shown using the same arguments as in the proof of Theorem~\ref{theorem1}.

\end{proof}

As a consequence of the positive definiteness of $V_T^{lim}$ and its monotonicity with respect to $T$, and analogously to Corollary~\ref{lemma-boundkinfty}, the following result follows.
\begin{corollary}
\label{lemma-boundkinfty-2}
Given $T>0$ there exist $\alpha_-,\alpha_+\in \mathcal{K}_\infty$ such that
\[\alpha_-(\|x_0\|)\leq V_T^{lim}(x_0) \leq \alpha_+(\|x_0\|).\]
Furthermore, the function $\alpha_+$ may be assumed to be independent of $T$, i.e., 
\[V_T^{lim}(x_0) \leq V_S^{lim}(x_0) \leq \alpha_+(\|x_0\|),\qquad\forall S\in (T,\infty].\]
\end{corollary}
\subsection{Study of the value function}
Similarly to the case $q<\infty$, one analyzes solutions of $(OCP)_T^{lim}$ with $T<\infty$ by using the PMP.
In particular, if $F$ satisfies \ref{ass1}, then similarly to Section~\ref{sec:hamiltonian}, solutions of $(OCP)_T^{lim}$ can be lifted to extremal pairs $(x^{(T)},p^{(T)})$ satisfying the Hamiltonian equations~\eqref{eq-Ham} with 
$H(x,u,p,\lambda) := p^\top J_n x +up_n - \lambda F(x)$
for some constant value $\lambda\in \{0,1\}$. Furthermore, the associated optimal control satisfies the minimization condition
\begin{equation*}
u^{(T)}(t)=\argmin_{u\in [-1,1]} H(x^{(T)}(t),u,p^{(T)}(t)),\quad \mbox{a.e. }t\in [0,T] 
\end{equation*}
and the following condition holds true
$p^{(T)}(T)=0$.
As in Section~\ref{sec:hamiltonian}, one has that, for optimal lifts, $\lambda=1$ and, from the minimization condition, the optimal control satisfies $u^{(T)}\in\lfloor  p^{(T)}_n\rceil^0.$
In particular, optimal pairs are solutions of
\begin{equation}
\begin{aligned}\label{eq-Ham3}
   \dot{x}& \in  J_n x+ \lfloor p_n\rceil^0 e_n, \quad
    \dot{p} =-J_n^\top p+\nabla F(x).
\end{aligned}
\end{equation}
Similarly to Theorem~\ref{theorem2}, we establish below a result ensuring that, under Assumption~\ref{ass1bis}, solutions of $(OCP)_\infty^{lim}$ can also be lifted to extremal pairs (satisfying~\eqref{eq-Ham3}), and that such extremal pairs uniformly converge to the origin. Note that the optimal control does not necessarily converge to zero. We have the following.
\begin{theorem}
\label{theorem2bis}
Assume that $F$ satisfies \ref{ass1bis}. For $x_0\in\mathbb{R}^n$, let $u^{\infty}$ and $x^{\infty}$ be an optimal control and associated trajectory for $(OCP)_\infty^{lim}$ starting from $x_0$. Then 
\begin{enumerate}
\item $x^{\infty}$ admits a unique extremal lift $(x^{\infty},p^{\infty})$, i.e., a unique solution of the Hamiltonian dynamics \eqref{eq-Ham3}, and $u^{\infty}(t) \in  \lfloor p^{\infty}_n(t)\rceil^0$ for almost every $t\geq 0$.
\item There exists a $\mathcal{K}_{\infty}$-function $C(\cdot)$ such that $\|(x^{\infty},p^{\infty})\|_\infty \leq C(\|x_0\|)$ and, moreover, $\lim_{t\to\infty}(x^{\infty}(t),p^{\infty}(t)) = (0,0)$.
\end{enumerate}
\end{theorem}
The crucial difference of the proof of Theorem~\ref{theorem2bis}, compared to that of Theorem~\ref{theorem2}, 
lies in the arguments leading to a uniform bound of $p^{(T)}(0)$ associated with $(OCP)_T^{lim}$. We will make use of the following result.
\begin{lemma}
\label{lemma-pol}
Let $T>0$. For every $\sigma\in (0,T/n)$ there exists $\eta>0$ such that the polynomial $P_\beta(t) = \sum_{i=1}^n \beta_i t^{i-1}$, with $\beta=(\beta_1,\dots,\beta_n)\in\mathbb{R}^n$, satisfies $|P_\beta(t)|\geq \eta\|\beta\|$ on a closed subinterval of $[0,T]$ of length $\sigma$.
\end{lemma}
\begin{proof}
Define $\lambda(\beta) := \max_{\tau\in [0,T-\sigma]}\min_{t\in [0,\sigma]}|P_\beta(\tau+t)|$  and observe that 
$\lambda(\beta)>0$ for every $\beta\neq 0$. Indeed, in this case $P_\beta$ admits at most $n-1$ real zeros so that there exists a subinterval $(\frac{k-1}{n}T,\frac{k}{n}T)$ for some $k\in \{1,\dots,n\}$ which does not contain zeros of $P_\beta$. In particular, $|P_\beta|>0$ on every closed subinterval of $(\frac{k-1}{n}T,\frac{k}{n}T)$ of length $\sigma$, showing that $\lambda(\beta)>0$. Furthermore, $\lambda$ is clearly a continuous function. We deduce that $\lambda(\beta)\geq \eta$ for some $\eta>0$ and for every $\beta$ satisfying $\|\beta\| = 1$. The lemma follows.
\end{proof}
\begin{proof}[Proof of Theorem~\ref{theorem2bis}]
Let $(x^{(T)},p^{(T)})$ be the optimal pairs for $(OCP)_T^{lim}$ with initial condition $x_0$, and $u^{(T)}$ the corresponding optimal controls. As $u^{(T)}$ belongs to $L^{\infty}([0,T),[-1,1])$ then, up to extending by zero on $[T,\infty)$ there exists an increasing and unbounded sequence $\{T_k\}_{k\geq 1}$ such that $u^{(T_k)}$ weak-$^*$ converges to an input $u^*\in L^{\infty}(\mathbb{R}_{\geq 0},[-1,1])$. Furthermore, the trajectories $x^{(T_k)}$ converge, uniformly on compact intervals, to the trajectory $x^*$ of~\eqref{n-integrator} associated with the input $u^*$. Let us prove the optimality of such a trajectory. Let $T>0$. If $k$
 is large enough so that $T_k\geq T$ then
 \[\int_0^T F(x^{(T_k)}(t))\,dt\leq V_{T_k}^{lim}(x_0)\leq V_\infty^{lim}(x_0).\]
 As $F(x^{(T_k)}(\cdot))$ converges uniformly to $F(x^*(\cdot))$ on $[0,T]$, passing to the limit as $k$ goes to infinity on the left-hand side we obtain 
  $\int_0^T F(x^*(t))\,dt\leq V_\infty^{lim}(x_0)$.
  Finally, letting $T$ tend to infinity and using the monotone convergence theorem, 
   $\int_0^\infty F(x^*(t))\,dt\leq V_\infty^{lim}(x_0)$.
   Then $x^* = x^\infty$ and $u^* = u^\infty$. 
   Moreover, as in the proof of Theorem~\ref{theorem2}, thanks to Corollary~\ref{lemma-boundkinfty-2} we have that
   \begin{equation}
   \label{eq:x-estimate-limit}
   \|x^\infty\|_\infty\leq \alpha_-^{-1}(\alpha_+(\|x_0\|)), 
   \end{equation}
   for some $\mathcal{K}_\infty$-functions $\alpha_-,\alpha_+$, and $\lim_{t\to\infty}x^\infty(t)=0$. 
   We will next show that $p^{(T_k)}(0)$ tends, as $k$ goes to infinity, to a value $p_0$ such that $\|p_0\|$ is bounded by a $\mathcal{K}_\infty$-function of $\|x_0\|$. With no loss of generality, we will prove the existence of such a function on an interval of definition $[0, M]$, where $M>0$ can be chosen arbitrarily large. 
   
   Let us fix $T>4\alpha_-^{-1}(\alpha_+(M))n$. Up to taking $k$ large enough we may assume, from \eqref{eq:x-estimate-limit} and the convergence of $x^{(T_k)}$ to $x^\infty$ on $[0,T]$, that $|x^{(T_k)}(t)|\leq 2\alpha_-^{-1}(\alpha_+(\|x_0\|))$ for $t\in [0,T]$. Applying the variation of constants formula to~\eqref{eq-Ham3} we can write
   \[p_n^{(T_k)}(t) = e_n^\top e^{-J_n^\top t}p^{(T_k)}(0) + e_n^\top\int_0^t e^{-J_n^\top (t-s)} \nabla F(x^{(T_k)}(s))\,ds\]
   for $t\in [0, T]$.
  
   The first term on the right-hand side is a polynomial in $t$ that can be expressed, in the notation of Lemma~\ref{lemma-pol}, as $P_\beta(t)$ with $\beta_h= \frac{1}{(h-1)!} p^{(T_k)}_{n+1-h}(0)$ for $h=1,\dots,n$. The second term is bounded by $\Gamma(\|x_0\|)$ for some $\mathcal{K}_\infty$-function $\Gamma$  independently of $t\in [0,T]$ and of $k$, as it follows from the uniform bound of the trajectories $x^{(T_k)}$ over $[0,T]$ and the fact that $\nabla F(0)=0$.
   
   Let now fix $\sigma\in (4\alpha_-^{-1}(\alpha_+(M)),T/n)$. 
   We claim that, for $k$ large enough, $p_n^{(T_k)}$ admits zeros on every subinterval of $[0,T]$ of length $\sigma$. Indeed, if that were not the case then $\dot{x}_n^{(T_k)}$ would be constantly equal to $1$ or $-1$ on such a subinterval, for $k$ arbitrarily large, contradicting the fact that $\|x^{(T_k)}(t+\sigma)-x^{(T_k)}(t)\|\leq 2\max_{s\in [0,T]} |x^{(T_k)}(s)| \leq 4\alpha_-^{-1}(\alpha_+(M))$ for every $t\in [0, T-\sigma]$.
   
   As a consequence, for every interval $I\subset [0,T]$ of length $\sigma$, one has $\min_{t\in I} |P_\beta (t)| \leq \Gamma(\|x_0\|)$.
   We deduce from Lemma~\ref{lemma-pol} that $\|\beta\|\leq \Gamma(\|x_0\|)/\eta$ for some $\eta>0$, leading to the uniform bound $\|p^{(T_k)}(0)\|\leq \Gamma(\|x_0\|)(n-1)!/\eta$. 
  
   Up to extracting a subsequence, we have that $p^{(T_k)}(0)$ converges to $p_0\in \mathbb{R}^n$ such that $\|p_0\|\leq \Gamma(\|x_0\|)(n-1)!/\eta$. As the right-hand side of~\eqref{eq-Ham3} is upper semicontinuous and takes compact convex values, we obtain, according to \cite[Lemma 1, p.87]{filippov2013differential}, the convergence of  $(x^{(T_k)},p^{(T_k)})$ to an extremal pair $(x^{\infty},p^{\infty})$, and the corresponding optimal control satisfies $u^{\infty}(t) \in  \lfloor p^{\infty}_n\rceil^0(t)$ for almost every $t\geq 0$. This completes the proof of Item 1. The proof of Item 2 can be completed thanks to Corollary~\ref{lemma-boundkinfty-2} and by following the same arguments as in the proof of Theorem~\ref{theorem2}.
 \end{proof}

Similarly to Theorem~\ref{prop11} and Corollary~\ref{cor-feedback}, we have the following result.

\begin{theorem}
    \label{HJB-inf}
The value function  $V_\infty^{lim}$ of $(OCP)_\infty^{lim}$ satisfies the HJB equation
\begin{equation}
\nabla V_\infty^{lim}(x)^\top J_nx  -\left|\frac{\partial V_\infty^{lim}}{\partial x}\right| + F(x) = 0, \label{eq:HJB-inf}
\end{equation}
at every differentiability point $x\in \mathbb{R}^n$
of $V_\infty^{lim}$. Furthermore, for every initial condition $x_0$, there exists an optimal solution of $(OCP)_\infty^{lim}$ such that the corresponding control input
satisfies $u(t) \in -\lfloor \ell_n(t) \rceil^0$ with $\ell(t) \in \partial V_\infty^{lim}(x(t))$ for almost every $t\geq 0$.
\end{theorem}
\begin{proof}
In the following, $(OCP)_\infty^{q}$ denotes the optimal control problem introduced in Section~\ref{sec:prob}, stressing the dependence on $q>1$ and assuming the function $F$ to be independent of $q$. We denote the corresponding value functions as $V_\infty^{(q)}$.  

Fix $x_0\in \mathbb{R}^n$, $\varepsilon>0$, and $q>1$. Consider an optimal control $u^\infty$ and the corresponding trajectory $x^\infty$ for $(OCP)_\infty^{lim}$. By Theorem~\ref{theorem2bis}, $x^\infty(t)$ tends to zero as $t$ goes to infinity and then, according to Corollary~\ref{lemma-boundkinfty} and Theorem~\ref{theorem2}, there exists $T>0$ large enough such that $V_\infty^{(q)}(x^\infty(T))<\varepsilon$ and the optimal control $u^{(q)}$ of  $(OCP)_\infty^{q}$ with initial condition $x^\infty(T)$ satisfies $\|u^{(q)}\|_\infty \leq 1$. Then, taking
$u^\varepsilon(t)$ equal to $u^\infty(t)$ if $t\leq T$ and $u^{(q)}(t-T)$ if $t>T$, and letting $x^\varepsilon$ be the corresponding trajectory from $x_0$, we have for every $p\geq q$
\begin{align*}
V_\infty^{(p)}(x_0)& \leq  \int_0^T \left(\frac{|u^\infty(t)|^{p}}{p} + F(x^\infty(t)) \right)dt +  \int_T^\infty \left(\frac{|u^{\varepsilon}(t)|^{p}}{p} + F(x^\varepsilon(t)) \right)dt\\
&  \leq  \int_0^T \left(\frac{1}{p} + F(x^\infty(t)) \right)dt + V_\infty^{(q)}(x^\infty(T))\leq   V_\infty^{lim}(x_0) + \frac{T}{p} + \varepsilon.
\end{align*}
By letting $p$ tend to infinity and $\varepsilon$ to zero we obtain that $\limsup_{p\to\infty} V_\infty^{(p)}(x_0) \leq V_\infty^{lim}(x_0)$ for every $x_0\in \mathbb{R}^n$. Furthermore, the previous estimate ensures that the convex function 
$x\mapsto \sup_{p>q} V_\infty^{(p)}(x)$
takes values in $\mathbb{R}$ for every $x\in \mathbb{R}^n$. In particular, the family $\{V_\infty^{(p)}\mid p>q\}$ is uniformly bounded on compact sets. By~\cite[Theorem 10.6]{rockafellar1997convex}, such a family is also equi-Lipschitzian, hence, by the Ascoli-Arzel\`a theorem, there exists a sequence $(p_k)_{k\in\mathbb{N}}\subset [q,\infty)$ converging to infinity such that $(V_\infty^{(p_k)})_{k\in\mathbb{N}}$ uniformly converges to a convex function $\bar V\leq V_\infty^{lim}$ on compact sets. Similarly to~\cite[Theorem 25.7]{rockafellar1997convex},
one has that $\nabla V_\infty^{(p_k)}$ converges to $\nabla \bar V$ on differentiability points of $\bar V$. Since each $V_\infty^{(p_k)}$ satisfies \eqref{eq:HJB} with $q=p_k$, letting $k$ tend to infinity we obtain
\[
\nabla\bar V(x)^\top J_nx  -\left|\frac{\partial \bar V}{\partial x_n}(x)\right| + F(x) = 0
\]
for every differentiability point $x$ of $\bar V$. 
Recall now that solutions of 
$(OCP)_\infty^{p_k}$ are trajectories of $\dot x = f_k(x)$ where $f_k(x) := J_nx- e_n \lfloor \frac{\partial V_\infty^{(p_k)}}{\partial x_n}\rceil^{\frac{1}{p_k-1}}$, and consider the differential inclusion $\dot x\in \mathcal{F}(x)$
where $\mathcal{F}(x)$ is the closed and convex set equal to $ J_nx- e_n \lfloor e_n^\top \partial \bar V \rceil^0$ for each $x\in \mathbb{R}^n$. Since $x\mapsto \mathcal{F}(x)$ is an upper semi-continuous set-valued map, $\dot x\in \mathcal{F}(x)$ admits nontrivial solutions $x^*$ on time intervals of the type 

$[0,T)$ with $T>0$. Then $x^*$ is a solution of~\eqref{n-integrator} corresponding to a (measurable) control input $e_n^\top(\dot{x}^*-J_nx^*)$ that takes values in $\lfloor e_n^\top \partial \bar V(x^*(t)) \rceil^0\subset [-1,1]$ for a.e. $t\in [0,T)$, and can in particular be extended to $[0,\infty)$.  
Applying~\cite[Theorem 1, p.87]{filippov2013differential} one obtains that for every $\epsilon>0$ and $T>0$ there exists $k$ large enough such that for every solution $x^{(k)}$ of $(OCP)_\infty^{p_k}$  there exists a solution $x^{(k)*}$ of the differential inclusion
$\dot x\in \mathcal{F}(x)$
satisfying $\|x^{(k)*}(t)-x^{(k)}(t)\|\leq \epsilon$ for $t\in[0,T]$.  
By arbitrariness of $\epsilon$ and $T$, $\liminf_{k\to\infty} V_\infty^{(p_k)}(x_0) \geq V_\infty^{lim}(x_0)$
for every $x_0\in\mathbb{R}^n$. Then $V_\infty^{lim} = \bar V$ and~\eqref{eq:HJB-inf} holds true at differentiability points of $V_\infty^{lim}$. One deduces from the Ascoli-Arzel\`a theorem and \cite[Corollary 1, p.77]{filippov2013differential} that for every initial condition there exists a solution of $\dot x=\mathcal{F}(x)$ which is also optimal for $(OCP)_\infty^{lim}$. This concludes the proof of the theorem.
\end{proof}

\begin{remark}
Note that the previous proof also shows the uniform convergence of $V_\infty^{(q)}$ to $V_\infty^{lim}$ on compact sets.
\end{remark}
\begin{proposition}
    \label{prop13-limit}
    Let $W\in \mathcal{C}^1(\mathbb{R}^n,\mathbb{R})$ be a positive definite function.
    \begin{enumerate}
    \setlength{\itemsep}{5pt}
        \item If $W$ is a sub-solution of \eqref{eq:HJB-inf}, i.e., $W$ verifies
        $\nabla W(x)^\top J_nx -\left|\frac{\partial W}{\partial x_n}\right| \leq -F(x)$ for $x \in \mathbb{R}^n$, then $V_\infty^{lim}(x) \leq W(x)$ for 
        $x \in \mathbb{R}^n$.
        \item If $V_\infty^{lim}\in \mathcal{C}^1(\mathbb{R}^n,\mathbb{R})$ and $W$ is a super-solution of \eqref{eq:HJB-inf}, i.e., $W$ verifies
        $\nabla W(x)^\top J_nx -\left|\frac{\partial W}{\partial x_n}\right| \geq -F(x)$ for $x \in \mathbb{R}^n$, 
        then $V_\infty^{lim}(x) \geq W(x)$ for $x \in \mathbb{R}^n$.
    \end{enumerate}
    
\end{proposition}
\begin{proof}
The proof follows similar lines to those of Proposition~\ref{prop13}, hence we will only single out the main differences. 
If $W$ is a sub-solution of \eqref{eq:HJB-inf} then we consider a solution $x_W$ with initial condition $x_0$ of the differential inclusion $\dot x\in J_nx - e_n \left\lfloor \frac{\partial W}{\partial x_n}(x)\right\rceil^0$. 
Such a solution exists over $\mathbb{R}_{\geq 0}$ since the right-hand side is an upper semi-continuous set-valued function with sublinear growth taking compact and convex values.
As in the proof of Proposition~\ref{prop13}, one has
\begin{align*}
    V_\infty^{lim}(x_0) \leq \int_0^\infty F(x_W(s))\, ds \leq W(x_0).
\end{align*}
Conversely, assume that $W$ is a super-solution of \eqref{eq:HJB-inf}. By Theorem~\ref{HJB-inf} for every initial condition $x_0
\in\mathbb{R}^n$ there exists an optimal solution $x^*$ of $(OCP)_\infty^{lim}$ such that the corresponding optimal control $u^*$ satisfies $u^*(t)\in -\left\lfloor \frac{\partial V_\infty^{lim}}{\partial x_n}(x^*(t))\right\rceil^0$ for almost every $t\geq 0$. Then, 
\[\frac{d}{dt}(V_\infty^{lim}-W)(x^*(t)) = (\nabla V_\infty^{lim}(x^*(t))-\nabla W(x^*(t))^\top (J_n x^*(t)+e_n u^*(t))\]
for almost every $t\geq 0$.
Following the proof of Proposition~\ref{prop13}, we then obtain 
   \[\frac{d}{dt}(V_\infty^{lim}-W)(x^*(t)) \leq   - \left|\frac{\partial W}{\partial x_n}(x^*(t))\right|+ \frac{\partial W}{\partial x_n}(x^*(t)) u^*(t)\leq 0.\] 
By integrating the previous inequality over $\mathbb{R}_{\geq 0}$ we obtain $W(x_0)\leq V_\infty^{lim}(x_0)$ so that, by arbitrariness of $x_0$, the proposition is proved. 
\end{proof}
\begin{remark}
\label{rem:hJB-lim}
It follows from Proposition~\ref{prop13-limit} that, whenever the value function $V_\infty^{lim}$ is continuously differentiable, then it is the unique solution of \eqref{eq:HJB-inf} among positive definite $\mathcal{C}^1$ functions defined in $\mathbb{R}^n$.
\end{remark}

\subsection{Finite-time convergence}

We next introduce a homogeneity assumption that, as we will see, enforces the finite-time convergence of the solutions of $(OCP)_\infty^{lim}$.

\medskip

\begin{enumerate}[label= (A\arabic*)$'$, leftmargin=1cm]
\setcounter{enumi}{1}
    \item \label{ass3}  Let $d,\mu \in \mathbb{R}_{>0}$ such that $\mu < \frac{d}{n}$ and define $r\in \mathbb{R}_{>0}^n$ by $r_i =(n-i+1)\mu$ for $i=1,\dots,n$. Then $F$ is homogeneous of degree $d$ with respect to the family $(\delta^r_\epsilon)_{\epsilon>0}$, i.e., 
   
    $F(\delta_\epsilon^r(x)) = \epsilon^d F(x)$ for all $\epsilon > 0$ and $x \in \mathbb{R}^n$.
\end{enumerate}
\begin{remark}
The right-hand of \eqref{eq-Ham3} defines a (set-valued) vector field $\mathcal{F}_\infty$ and, similarly to Remark~\ref{rem:HF-homogeneity},  this Hamiltonian vector field is homogeneous with respect to a family of dilations defined 
as follows. Let $s=(s_1,\cdots,s_n)$ with 
$s_i=d-(n-i)\mu$ for $1\leq i\leq n$.
Then $\mathcal{F}_\infty$ is homogeneous of degree $-\mu$ with respect to the family of dilations  
    $(\delta^{(r,s)}_\epsilon)_{\epsilon > 0}$, where $\delta^{(r,s)}_\epsilon(x,p):= (\delta^r_\epsilon(x), \delta^s_\epsilon(p))$ is defined on $\mathbb{R}^n\times\mathbb{R}^n$. 
\end{remark}
\noindent 
Similarly to Lemma~\ref{lem-hom-opt-0}, one gets that,
under Assumption~\ref{ass3},
applying the dilation $\delta_\epsilon^r$ to a solution of $(OCP)_\infty^{lim}$ preserves its 
optimality, up to a time rescaling.
\begin{lemma}
\label{lem-hom-opt}
Assume that $F$ satisfies \ref{ass3}, in addition to \ref{ass1}. Let $ x^\infty(\cdot)$ be an optimal trajectory for $(OCP)_\infty^{lim}$. Then, for every $\epsilon>0$, the map $x^\infty_\epsilon(t) := \delta^r_\epsilon (x^\infty(\epsilon^{-\mu}t))$ also corresponds to  an optimal trajectory for $(OCP)_\infty^{lim}$.  Moreover $V_\infty^{lim}$ is homogeneous of degree $d +\mu$ with respect to  $(\delta^r_\epsilon)_{\epsilon>0}$. 
\end{lemma}

Using the previous lemma, the proof of Theorem~\ref{theorem-finitetime} can be adapted to solutions of $(OCP)_\infty^{lim}$ in a straightforward manner. Hence, we get the following result.
\begin{theorem}
    \label{theorem-finitetime-limit}
    Assume that $F$ satisfies \ref{ass3}, in addition to \ref{ass1}. Let  $x_0\in \mathbb{R}^n$ and $x^\infty$ be an optimal trajectory for $(OCP)_\infty^{lim}$ starting at $x_0$. Then $x^{\infty}$ reaches the origin in finite time. Moreover, the settling-time map $x_0\mapsto T_{x_0}$ is 
    homogeneous of degree $\mu$ with respect to $(\delta_\epsilon^r)_{\epsilon>0}$ and 
    $T_{x_0}\leq \frac{d+\mu}{c\mu}\,V_{\infty}^{lim}(x_0)^\frac{\mu}{d+\mu}$ for every $x_0\in\mathbb{R}^n$, where $c = \max_{x \in (V_\infty^{lim})^{-1}(1)} F(x)$.

\end{theorem}
We next show that, similarly to Proposition~\ref{cor3}, solutions of $(OCP)^{lim}_\infty$ are also solutions of $(OCP)^{lim}_T$ for $T$ large enough, and we deduce from this fact that $V^{lim}_\infty$ is continuously differentiable.
\begin{proposition}
Assume that $F$ satisfies \ref{ass1bis} and \ref{ass3}. Let $x^{\infty}$ and $T_{x_0}$ as in Theorem~\ref{theorem-finitetime-limit} and denote as $(x^{\infty},p^{\infty})$ the associated extremal lift.
Then, $(x^{\infty}(t),p^{\infty}(t))=(0,0)$ for every $t\geq T_{x_0}$. Moreover, for every $T\geq T_{x_0}$ the restriction of $(x^{\infty},p^{\infty})$ to $[0, T]$ corresponds to the unique extremal lift associated with the solution of $(OCP)_{T}^{lim}$. Furthermore, the value function $V_\infty^{lim}$ is $\mathcal{C}^1$.    
\end{proposition}
\begin{proof}
The first part of the proposition is obtained by using the same arguments as in the proof of Proposition~\ref{cor3}. Concerning the regularity of $V_\infty^{lim}$, note that the optimal control problem~$(OCP)_T^{lim}$ can be recast within the framework discussed in~\cite{goebel2004regularity}, and in particular the function $V_T^{lim}$ is continuously differentiable thanks to~\cite[Theorem 4.6]{goebel2004regularity}. By the first part of the proposition it follows that $V_\infty^{lim}(x_0)=V_T^{lim}(x_0)$ for $T\geq T_{x_0}$. In particular, if we choose any $T>\frac{d+\mu}{c\mu}\,{V_{\infty}^{lim}}(x_0)^\frac{\mu}{d+\mu}$ where $c = \max_{x \in (V_\infty^{lim})^{-1}(1)} F(x)$ then, according to Theorem~\ref{theorem-finitetime} and by continuity of $V_\infty^{lim}$, one has that $T>T_x$ for $x$ in a small enough neighborhood $U$ of $x_0$. Hence $V_\infty^{lim}=V_T^{lim}$ on $U$ proving the continuous differentiability of $V_\infty^{lim}$ at $x_0$ and therefore, by arbitrariness of $x_0$, on the whole $\mathbb{R}^n$.
\end{proof}

\begin{remark}
    By the previous result and 
    Remark~\ref{rem:hJB-lim}, if $F$ satisfies \ref{ass1bis} and \ref{ass3} then the value function satisfies the following Hamilton-Jacobi-Bellman equation
    \[\nabla V_\infty^{lim}(x)^\top J_nx -\left|\frac{\partial V_\infty^{lim}}{\partial x_n}(x)\right| + F(x) = 0,\qquad \forall x \in \mathbb{R}^n. \]
    On the other hand, similarly to Proposition~\ref{cor7}, if $W\in\mathcal{C}^1(\mathbb{R}^n,\mathbb{R})$ is positive definite and such that 
    $F_W(x):= \left|\frac{\partial W}{\partial x_n}(x)\right|-\nabla W(x)^\top J_n x$
    is a set-valued map bounded from below and above by positive definite functions homogeneous of degree $d$ with respect to $(\delta_\epsilon^r)_{\epsilon>0}$ then one can show that the trajectories of~\eqref{n-integrator} whose control input satisfies $u(t)\in -\left\lfloor\frac{\partial W}{\partial x_n}(x(t))\right\rceil^0$ for almost every $t\geq 0$ converge to the origin in finite time. 
\end{remark}
\begin{remark}
\label{rk:sing-arc}
A numerical analysis of the optimal control problem $(OCP)_{T}^{lim}$ for $F$ satisfying the homogeneity assumption \ref{ass3} suggests that optimal trajectories are, generically, concatenations of bang and singular arcs. Recall that bang arcs correspond to time intervals where $p_n\neq 0$ (hence the optimal control is identically equal to 1 or -1), while a singular arc corresponds to an interval where $p_n$ is identically equal to zero. In particular, the value of the control input along singular arcs can be obtained by imposing $p_n$ and its derivatives equal to zero. If, for simplicity, we assume $F\in \mathcal{C}^2$ with Hessian $\nabla^2 F$ positive definite, this yields the singular control
\begin{equation}
    \label{eq:sing}
    u_{sing} = \frac{\nabla F(x)^\top J_ne_n -p^\top J^2_ne_n-e_n^\top \nabla^2 F(x)J_n x}{e^\top_n\nabla^2F(x) e_n},
\end{equation}
in case the latter expression belongs to $[-1,1]$, and with the additional constraints $p_n=0$ and $e_n^\top J_n^\top p = \frac{\partial F}{\partial x_n}(x)$. 
\end{remark}

\section{Numerical implementations}
\label{sec:simu}
In this section, we present some numerical simulations that illustrate the finite-time convergence of optimal trajectories of $(OCP)_\infty$ and $(OCP)_\infty^{lim}$ under suitable homogeneity assumptions, as established in Theorem~\ref{theorem-finitetime} and Theorem~\ref{theorem-finitetime-limit}.
 
Specifically, we display the graphs of some optimal trajectories and corresponding controls with $q=2$ and $q=\infty$, assuming that the function $F$ in the instantaneous cost satisfies the homogeneity assumption~\ref{ass2} and \ref{ass3}, respectively. %selected parameters $d$ and $\mu$ matching the choice of the parameter $q$. 
The simulations are performed using the \textit{Julia} packages NLPModelsIpopt.jl and OptimalControl.jl~\cite{Caillau_OptimalControl_jl_a_Julia}. %\cite[Packages: OptimalControl.jl, NLPModelsIpopt]{Caillau_OptimalControl_jl_a_Julia}.
We first consider a chain of integrators of length six with $q=2$, $d=1$, $\mu=0.05$, and the cost function $F$ given by
\begin{equation}
\label{F-sim-1}
F(x):= |x_1|^\frac{5}{4}+|x_2|^\frac{4}{3}+|x_3|^\frac{10}{7}+|x_4|^\frac{20}{13}+|x_5|^\frac{5}{3}+|x_6|^\frac{20}{11},\qquad x\in \mathbb{R}^6.
\end{equation}
Note that the exponents of $F$ are obtained as the inverse of the weights $r_i$ in~\ref{ass2}.
Fig.~\ref{fig1} displays an optimal trajectory and the corresponding control input computed via a direct method, with a time discretization step equal to 0.12. 
\begin{figure}[h]
    \centering
    \resizebox{0.7\textwidth}{0.3\textheight}{
    \includegraphics{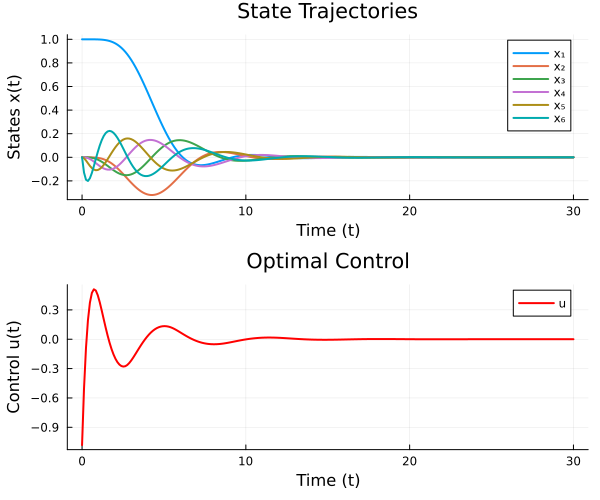}
    }
    \caption{Optimal state and control associated with $q=2$ and with $F$ given by~\eqref{F-sim-1} starting from (1,0,0,0,0,0) with a direct method.}
    \label{fig1}
\end{figure}
\FloatBarrier
For $q=+\infty$, we consider a chain of integrators of length four. Take $d=8$, $\mu=1$ and the cost function given by 
\begin{equation}
\label{F-sim-2}
F(x) :=x_1^2 + |x_2|^\frac{8}{3} + x_3^4 + x_4^8,\qquad x\in \mathbb{R}^4.
\end{equation}
Note that the exponents of $F$ are obtained as $d/r_i$,  
$r_i$ being the weights in~\ref{ass3}. Notice that $F$ is twice differentiable, so that the Hessian matrix appearing in~\eqref{eq:sing} is well-defined. 
In accordance with Remark~\ref{rk:sing-arc}, the direct method generates an optimal control which appears to be a concatenation of bang and singular arcs. Optimal controls computed via the direct method exhibit an oscillating behavior in correspondence with singular arcs; such oscillations vanish for a sufficiently small time discretization step, although spikes (interpretable as artifacts of the direct method) appear at switching times between bang and singular arcs. Alternatively, optimal trajectories can be computed numerically via an indirect shooting method which relies on \eqref{eq:sing} on singular arcs, and is initialized with the switching times and values of the variable $p$ obtained thanks to the direct method.
Fig.~\ref{fig2} displays the reconstruction of an optimal trajectory and the corresponding control by means of both direct and indirect methods, showing a precise matching between the corresponding solutions; the time discretization step has been taken equal to 0.0027.

\begin{figure}[h]
    \centering
    \resizebox{0.9\textwidth}{0.4\textheight}{
        \includegraphics{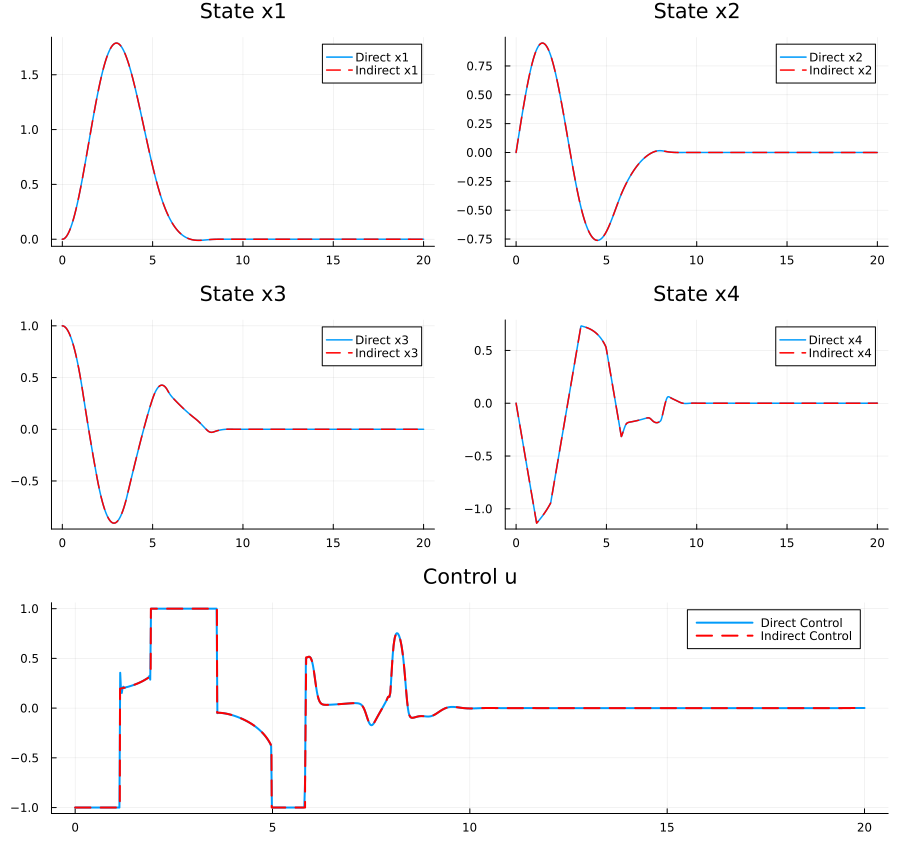}
    }
    \caption{Optimal state and control associated with $q=\infty$ 
    and with $F$ given by~\eqref{F-sim-2} starting from (0,0,1,0) obtained with a direct method and with an indirect shooting method.}
    \label{fig2}
\end{figure}
\FloatBarrier

\bibliographystyle{siamplain}
\bibliography{references}
\end{document}